\documentclass[12pt]{article}
\usepackage[top=1in, bottom=1in, left=1in, right=1in]{geometry}
\usepackage{amssymb}
\usepackage{float,epsfig, floatflt,here}
\usepackage{pgf,tikz,amsmath}
\usepackage{graphicx} 
\usepackage{blkarray}
\usepackage{caption}
\usepackage{enumerate}
\usetikzlibrary{arrows}
\usepackage{bbm} 

\usepackage[pagebackref,colorlinks=true,
linkcolor=blue,
urlcolor=red,colorlinks,
citecolor=green]{hyperref}
\usepackage{cleveref}
\usepackage{amsthm}
\usetikzlibrary{decorations.markings}
\tikzset{->-/.style={decoration={
			markings,
			mark=at position #1 with {\arrow{>}}},postaction={decorate}}}

\usepackage{tikz}
\usepackage{tikzit}

\tikzstyle{Boxy}=[fill=black, draw=black, shape=rectangle]
\tikzstyle{doty}=[fill=black, draw=black, shape=circle, tikzit fill=black, tikzit draw=black, tikzit shape=circle]

\tikzstyle{new edge style 0}=[-]
\tikzstyle{Line}=[-, draw={rgb,255: red,128; green,128; blue,128}]

\usetikzlibrary{calc}
\usetikzlibrary{patterns}

\usepackage{subcaption}

 \newtheorem{theorem}{Theorem}
\newtheorem{remark}[theorem]{Remark}
\newtheorem{example}[theorem]{Example}
\newtheorem{lemma}{Lemma}
\newtheorem{corollary}{Corollary}
\newtheorem{definition}{Definition}
\newtheorem{proposition}{Proposition}
\newtheorem{note}{Note}

\newtheorem{case}{Case}

\def \bd{\begin{definition}}
	\def \ed{\end{definition}}
\def \bt{\begin{theorem}}
	\def \et{\end{theorem}}
\def \bl{\begin{lemma}}
	\def \el{\end{lemma}}
\def \bc{\begin{corollary}}
	\def \ec{\end{corollary}}
\def \be{\begin{equation}}
\def \ee{\end{equation}}
\def \ba{\begin{array}}
	\def \ea{\end{array}}
\def\bp{\begin{proposition}}
	\def\ep{\end{proposition}}
\def \bx{\begin{example}}
	\def \ex{\end{example}}
\def\bxa{\begin{example}\rm}
	\def\exa{\end{example}}
\def \br{\begin{remark}}
	\def \er{\end{remark}}
\def \bdsc{\begin{description}}
	\def \edsc{\end{description}}
\def \bn{\begin{case}}
	\def \en{\end{case}}
\def \bfig{\begin{figure}}
	\def \efig{\end{figure}}
\def \bpic{\begin{picture}}
\def \epic{\end{picture}}
\def \bnn{\begin{note}}
	\def \enn{\end{note}}

\newcommand{\bea}{\begin{eqnarray}}
\newcommand{\eea}{\end{eqnarray}}
\newcommand{\ben}{\begin{eqnarray*}}
	\newcommand{\een}{\end{eqnarray*}}
\newcommand{\NI}{\noindent}

\def\pf{\textbf {Proof: }}

\def\1{1\!\!1}
\def\0{0\!\!0}

\title{Eccentric graph of trees and their Cartesian products}
\author{Anita Arora\footnote{Department of Mathematics, Indian Institute of Science, Bangalore, India (anitaarora@iisc.ac.in).} 
\and Rajiv Mishra\footnote{Department of Mathematics and Statistics, IISER Kolkata, Kolkata, India (rm20rs017@iiserkol.ac.in).}
}

\begin{document}	
\maketitle
\begin{abstract}
Let $G$ be an undirected simple connected graph. We say a vertex $u$ is eccentric to a vertex $v$ in $G$ if $d(u,v)=\max\{d(v,w): w\in V(G)\}$. The eccentric graph, $E(G)$ of $G$ is a graph defined on the same vertex set as of $G$ and two vertices are adjacent if one is eccentric to the other. 
We find the structure and the girth of the eccentric graph of trees
and see that the girth of the eccentric graph of a tree can either be zero, three, or four. Further, we study the structure of the eccentric graph of the Cartesian product of graphs and prove that the girth of the eccentric graph of the Cartesian product of trees can only be zero, three, four or six. Furthermore, we provide a comprehensive classification when the eccentric girth assumes these values. We also give the structure of the eccentric graph of the grid graphs and the Cartesian product of cycles.
Finally, we determine the conditions under which the eccentricity matrix of the Cartesian product of trees becomes invertible.
\end{abstract}



\section{Introduction} 
\label{sec: basicEG}
Let $G$ be a simple undirected graph on $n$ vertices with $m$ edges and $V(G)$ denote the set of vertices in $G$. If two vertices $v,w\in V(G)$ are adjacent, we will write $v\sim_{G} w$.  
The \emph{neighbourhood} of a vertex $v$ in $G$ is defined as $N_G(v)=\{w:v\sim_{G} w\}$. 
If the graph $G$ is connected,
the \emph{distance} $d_G(v, w)$, between two vertices $v$ and $w$ is the length of the shortest path in $G$ connecting
them.  The \emph{distance matrix} of a connected graph $G$, denoted
as $D(G)$, is the $n\times n$ matrix indexed by $V(G)$ whose $(v,w)$th-entry is equal to $d_G(v, w)$. We will only consider simple, undirected graphs on \emph{at least} two vertices in this paper. 

The \emph{eccentricity}, $e_G(v)$, of a vertex $v\in V(G)$ is defined as $$e_G(v)=\max\{d(u,v): u\in V(G)\},$$
we will use $e(v)$ instead of $e_G(v)$ whenever there is no confusion about the underlying graph.
If $d(u,v)=e(v)$, then we will say $u$ is \emph{eccentric} to $v$ and a shortest path between $u$ and $v$ is called an \emph{eccentric path} (starting from $v$). The \emph{diameter} of $G$, $diam(G)$, is the maximum of eccentricities of the vertices in $G$. A \emph{diametrical path} is a longest path among all eccentric paths in the graph $G$.

The \emph{eccentricity matrix} of a connected graph $G$, denoted by $\mathcal{E}_G$, is constructed from the distance matrix $D(G)$, retaining the largest distances in each row and each column, while other elements of the distance matrix are set to zero.
In other words,
$$(\mathcal{E}_G)_{i\,j}=\begin{cases}
    d(u_i,u_j) & \text{ if } d(u_i,u_j)=\min\{e(u_i),e(u_j)\},\\
    0 & \text{ otherwise.}
\end{cases}$$

\bd []\rm 
The \emph{eccentric graph} $E(G)$, of a connected graph $G$ is the simple graph with the vertex set same as that of $G$ and $(v,u)$ is an edge in $E(G)$ if either $v$ is eccentric to $u$ or $u$ is eccentric to $v$. If $u$ is adjacent to $v$ in $E(G)$, we will write $u\sim_{\scriptscriptstyle E(G)} v $.
\ed
{Note that the adjacency matrix of the eccentric graph $E(G)$ is obtained by replacing the non-zero entries in the eccentricity matrix $\mathcal{E}_G$, by 1.}

Recall the girth of a graph $G$ is the length of the shortest cycle present in $G$. {If a graph $G$ has no cycles, we will say that $G$ has girth 0.} We will call the girth of the eccentric graph as \emph{eccentric girth}. {Girth is the dual concept to edge connectivity, in the sense that the girth of a planar graph is the edge connectivity of its dual graph, and vice versa. Calculating the girth of a graph is an important task in graph theory, as it helps us understand the graph's structure and properties.}

The notion of eccentricity matrix was first introduced by Randi$\grave{c}$
as the $D_{max}$-matrix in 2013 \cite{Dmax2013} and subsequently, Wang et al. renamed it as the eccentricity matrix in 2018~\cite{EccMatrix2018}.
The eccentricity matrix of a graph is also called as anti-adjacency matrix in the following sense. The eccentricity matrix is obtained
from the distance matrix by preserving only the largest distances in each row and column; on the other hand, the adjacency matrix is obtained from the
distance matrix by preserving only the smallest non-zero distances in each row and column. Unlike the
adjacency matrix and the distance matrix, the eccentricity matrix  of a connected graph need not be irreducible. The eccentricity matrix of a complete bipartite graph is reducible and the eccentricity matrix of a tree is irreducible~\cite{EccMatrix2018,Mahato2020}. 

Spectra of the eccentricity matrix for some graphs are studied by Mahato et al.~\cite{Mahato2020} and Wang et al.~\cite{EccMatrix2018}, the lower and upper bounds for the $\mathcal{E}$-spectral radius of graphs are also discussed in \cite{EccMatrix2018}. 
{J. Wang et al. studied the non-isomorphic co-spectral graphs with respect to the eccentricity matrix  \cite{wang2022spectral}}.
Eccentricity matrix has interesting applications, its main application is in the field of chemical graph theory ~\cite{Dmax2013,DmaxApplication2013}.

A necessary and sufficient condition for $E(G)$ to be isomorphic to $G$ or the complement of $G$ is given by Akiyama et al.~\cite{EccGraph1985}. Kaspar et al. gave complete structure of the eccentric graph for some well-known graphs like paths and cycles~\cite{PathCycleEG2018}.
A \emph{star graph} $S_{n}$, on $(n+1)$ vertices is a graph with $n$ vertices of degree 1 and one vertex, called the center, of degree $n$. A \emph{double star} $S_{s,t}$, is a graph obtained by adding an edge between the center vertices of two stars, $S_{s}$ and $S_{t}$.
Let $P_n$ denotes the \emph{path} graph on $n$ vertices with the natural labelling $1,2,\ldots,n$. Then, 
\[ E(P_n)=
\begin{cases}
K_n,\text{ the complete graph }& \text{ if }n\leq 3,\\
S_{\frac{n-2}{2},\frac{n-2}{2}},\text{ a double star }& \text{ if }n>3 \text{ and is even},\\
H_{\frac{n-3}{2}}& \text{ if }n>3 \text{ and is odd},
\end{cases}
\]
where $K_t$ denotes the complete graph on $t$ vertices and $H_t$ is a graph obtained by adding $t$ pendant vertices to each of any two of the vertices of a triangle (see \Cref{fig:EGofP8andP9}).
\begin{figure}[h!]
    \centering
    \begin{tikzpicture}[x=0.7 cm,y=0.7 cm]
	\begin{pgfonlayer}{nodelayer}
		\node [style=doty, scale=0.6] (0) at (-5, 2) {};
		\node [style=doty, scale=0.6] (1) at (-3, 2) {};
		\node [style=doty, scale=0.6] (2) at (-6.25, 3.25) {};
		\node [style=doty, scale=0.6] (3) at (-6.75, 2) {};
		\node [style=doty, scale=0.6] (4) at (-6.25, 0.75) {};
		\node [style=doty, scale=0.6] (5) at (-1.75, 3.25) {};
		\node [style=doty, scale=0.6] (6) at (-1.25, 2) {};
		\node [style=doty, scale=0.6] (7) at (-1.75, 0.75) {};
		\node [style=doty, scale=0.6] (8) at (1.25, 2) {};
		\node [style=doty, scale=0.6] (9) at (3, 2) {};
		\node [style=doty, scale=0.6] (10) at (5, 2) {};
		\node [style=doty, scale=0.6] (11) at (6.75, 2) {};
		\node [style=doty, scale=0.6] (12) at (1.75, 3.25) {};
		\node [style=doty, scale=0.6] (13) at (1.75, 0.75) {};
		\node [style=doty, scale=0.6] (14) at (6.25, 3.25) {};
		\node [style=doty, scale=0.6] (15) at (6.25, 0.75) {};
		\node [style=doty, scale=0.6] (16) at (4, 3.5) {};
	\end{pgfonlayer}
	\begin{pgfonlayer}{edgelayer}
		\draw [line width=1pt](0) to (1);
		\draw [line width=1pt](1) to (5);
		\draw [line width=1pt](1) to (6);
		\draw [line width=1pt](1) to (7);
		\draw [line width=1pt](2) to (0);
		\draw [line width=1pt](3) to (0);
		\draw [line width=1pt](4) to (0);
		\draw [line width=1pt](9) to (10);
		\draw [line width=1pt](10) to (11);
		\draw [line width=1pt](9) to (8);
		\draw [line width=1pt](9) to (13);
		\draw [line width=1pt](12) to (9);
		\draw [line width=1pt](9) to (16);
		\draw [line width=1pt](16) to (10);
		\draw [line width=1pt](10) to (14);
		\draw [line width=1pt](10) to (15);
	\end{pgfonlayer}
\end{tikzpicture}
    \caption{Eccentric graph of the path graphs $P_8$,  and $P_9$ ($S_{3,3}$ and $H_3$).}
    \label{fig:EGofP8andP9}
\end{figure}
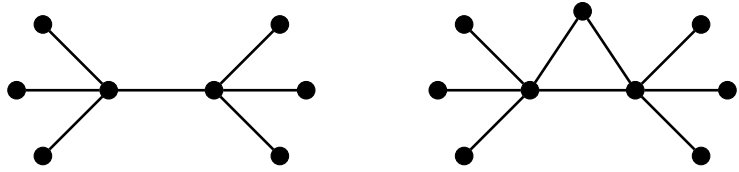

Let $C_n$ denotes the cycle graph on $n$ vertices and the vertices are labeled as $1,2,\ldots,n$. Then,
\begin{equation}
\label{eqn:eccGraphOfCycle}
E(C_n)=
\begin{cases}
\frac{n}{2}K_2 &\text{ if } n \text{ is even,}\\
C_n &\text{ if } n \text{ is odd.}
\end{cases} 
\end{equation}
Also, $E(K_n)=K_n$ and $E(K_{s,t})=K_s\cup K_t$ for $s,t>1$~\cite{PathCycleEG2018}. Throughout the paper, we will use the notation $P_n$ and $C_n$ to denote the path graph and the cycle graph on $n$ vertices.

Numerous interesting properties of the eccentric matrix of a tree have been established so far. For instance, Mahato showed that the eccentric matrix of a tree is invertible only if the tree is a star~\cite{Mahato-2023}. Additionally, the diameter of the tree is odd if and only if eigenvalues of its eccentric matrix are symmetric about the origin~\cite{MahatoSymmetry2022}. 

In \Cref{sec:Structure of $E(T)$}, we will give a complete structure of the eccentric graph of a general tree
and point out one more structural information in \Cref{prop:Ecc_is_smallest_or_largest_in_tree}. In
\Cref{sec:EccGir of a tree}, we will prove that the eccentric girth of a tree can either be zero, three or four.
In \Cref{sec:EccGofCProd}, we will present some structural properties of the eccentric graph of the Cartesian product of graphs and classify all the possible values of the eccentric girth of the Cartesian product of trees.

Lastly, in \cref{sec:Invertibilty of EG}, generalising the result of Mahato~\cite[Theorem 2.1]{Mahato-2023}, we will analyze and classify the conditions under which the eccentricity matrix of the Cartesian product of trees becomes invertible. 

\section{Structure of eccentric graph of a tree}
\label{sec:Structure of $E(T)$}

In this section, we will focus on the structure of the eccentric graph of a tree. Recall that a \emph{tree} is a connected graph with no cycles and the \emph{degree} of a vertex $v$ in a simple graph $G$ is the number of vertices adjacent to it. A vertex of degree $1$ is called a \emph{leaf} or a \emph{pendant} vertex. The \emph{union} of two graphs $G_1$ and $G_2$ is the simple graph whose vertex set and edge set are formed by taking the union of the vertex sets {of $G_1$ and $G_2$} and the edge sets of $G_1$ and $G_2$, respectively.

\bd
Let $T$ be a tree and $v$ be a leaf vertex in $T$. We define the path from $v$ to the nearest vertex of degree greater than two as the \emph{stem} at $v$.
\ed
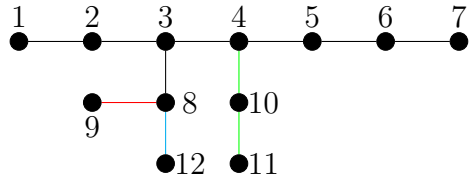
\begin{figure}[H]
    \centering
\begin{tikzpicture}[x=0.65cm, y=0.65cm]
	\begin{pgfonlayer}{nodelayer}
		\node [style=doty, scale=0.6] (0) at (-2, 1) {};
		\node [style=doty, scale=0.6] (1) at (-0.5, 1) {};
		\node [style=doty, scale=0.6] (2) at (1, 1) {};
		\node [style=doty, scale=0.6] (3) at (2.5, 1) {};
		\node [style=doty, scale=0.6] (4) at (4, 1) {};
		\node [style=doty, scale=0.6] (5) at (5.5, 1) {};
		\node [style=doty, scale=0.6] (6) at (7, 1) {};
		\node [style=doty, scale=0.6] (7) at (1, -0.25) {};
		\node [style=doty, scale=0.6] (8) at (-0.5, -0.25) {};
		\node [style=doty, scale=0.6] (9) at (2.5, -0.25) {};
		\node [style=doty, scale=0.6] (10) at (1, -1.5) {};
		\node [style=doty, scale=0.6] (11) at (2.5, -1.5) {};
		\node [style=none] (12) at (-2, 1.5) {1};
		\node [style=none] (13) at (-0.5, 1.5) {2};
		\node [style=none] (14) at (1, 1.5) {3};
		\node [style=none] (15) at (2.5, 1.5) {4};
		\node [style=none] (16) at (4, 1.5) {5};
		\node [style=none] (17) at (5.5, 1.5) {6};
		\node [style=none] (18) at (7, 1.5) {7};
		\node [style=none] (19) at (1.5, -0.25) {8};
		\node [style=none] (20) at (3, -0.25) {10};
		\node [style=none] (21) at (1.5, -1.5) {12};
		\node [style=none] (22) at (3, -1.5) {11};
		\node [style=none] (23) at (-0.5, -0.75) {9};
	\end{pgfonlayer}
	\begin{pgfonlayer}{edgelayer}
		\draw (0) to (1);
		\draw (1) to (2);
		\draw (2) to (3);
		\draw (3) to (4);
		\draw (4) to (5);
		\draw (5) to (6);
		\draw (2) to (7);
		\draw [color=red](7) to (8);
		\draw [color=cyan](7) to (10);
		\draw [color=green](3) to (9);
		\draw [color=green](9) to (11);
	\end{pgfonlayer}
\end{tikzpicture}
\caption{A tree $T$ on 12 vertices with different colored stems at vertices 9, 11 and 12.}
\label{fig:A tree}
\end{figure}

\emph{Note that a path graph $P_n$ has no stems. }
\begin{definition}
\label{def: induced tree from diam path}

    Let $P$ be a diametrical path in a tree $T$. We define the \emph{tree induced from the path $P$} as the subtree of $T$ obtained by removing stems at those leaves 
    {(except endpoints of $P$)}, which are an endpoint of some diametrical path other than $P$.
\end{definition}

Consider the tree $T$ shown in \Cref{fig:A tree}. $T$ has three diametrical paths and the subtrees induced by these are shown in \Cref{fig:3subtreesofT}.
\begin{figure}[H]
\begin{tikzpicture}[x=0.6 cm,y=0.6 cm]
	\begin{pgfonlayer}{nodelayer}
		\node [style=doty, scale=0.6] (0) at (-4.5, 1) {};
		\node [style=doty, scale=0.6] (1) at (-3, 1) {};
		\node [style=doty, scale=0.6] (2) at (-1.5, 1) {};
		\node [style=doty, scale=0.6] (3) at (0, 1) {};
		\node [style=doty, scale=0.6] (4) at (1.5, 1) {};
		\node [style=doty, scale=0.6] (5) at (3, 1) {};
		\node [style=doty, scale=0.6] (6) at (4.5, 1) {};
		\node [style=doty, scale=0.6] (7) at (-1.5, -0.25) {};
		\node [style=doty, scale=0.6] (9) at (0, -0.25) {};
		\node [style=doty, scale=0.6] (11) at (0, -1.5) {};
		\node [style=none, scale=0.8] (12) at (-4.5, 1.5) {1};
		\node [style=none, scale=0.8] (13) at (-3, 1.5) {2};
		\node [style=none, scale=0.8] (14) at (-1.5, 1.5) {3};
		\node [style=none, scale=0.8] (15) at (0, 1.5) {4};
		\node [style=none, scale=0.8] (16) at (1.5, 1.5) {5};
		\node [style=none, scale=0.8] (17) at (3, 1.5) {6};
		\node [style=none, scale=0.8] (18) at (4.5, 1.5) {7};
		\node [style=none, scale=0.8] (19) at (-1, -0.25) {8};
		\node [style=none, scale=0.8] (20) at (0.5, -0.25) {10};
		\node [style=none, scale=0.8] (22) at (0.5, -1.5) {11};
		\node [style=doty, scale=0.6] (26) at (8, 1) {};
		\node [style=doty, scale=0.6] (27) at (9.5, 1) {};
		\node [style=doty, scale=0.6] (28) at (11, 1) {};
		\node [style=doty, scale=0.6] (29) at (12.5, 1) {};
		\node [style=doty, scale=0.6] (30) at (14, 1) {};
		\node [style=doty, scale=0.6] (31) at (8, -0.25) {};
		\node [style=doty, scale=0.6] (32) at (6.5, -0.25) {};
		\node [style=doty, scale=0.6] (33) at (9.5, -0.25) {};
		\node [style=doty, scale=0.6] (35) at (9.5, -1.5) {};
		\node [style=none, scale=0.8] (38) at (8, 1.5) {3};
		\node [style=none, scale=0.8] (39) at (9.5, 1.5) {4};
		\node [style=none, scale=0.8] (40) at (11, 1.5) {5};
		\node [style=none, scale=0.8] (41) at (12.5, 1.5) {6};
		\node [style=none, scale=0.8] (42) at (14, 1.5) {7};
		\node [style=none, scale=0.8] (43) at (8.5, -0.25) {8};
		\node [style=none, scale=0.8] (44) at (10, -0.25) {10};
		\node [style=none, scale=0.8] (46) at (10, -1.5) {11};
		\node [style=none, scale=0.8] (47) at (6.5, -0.75) {9};
		\node [style=doty, scale=0.6] (50) at (16.5, 1) {};
		\node [style=doty, scale=0.6] (51) at (18, 1) {};
		\node [style=doty, scale=0.6] (52) at (19.5, 1) {};
		\node [style=doty, scale=0.6] (53) at (21, 1) {};
		\node [style=doty, scale=0.6] (54) at (22.5, 1) {};
		\node [style=doty, scale=0.6] (55) at (16.5, -0.25) {};
		\node [style=doty, scale=0.6] (57) at (18, -0.25) {};
		\node [style=doty, scale=0.6] (58) at (16.5, -1.5) {};
		\node [style=doty, scale=0.6] (59) at (18, -1.5) {};
		\node [style=none, scale=0.8] (62) at (16.5, 1.5) {3};
		\node [style=none, scale=0.8] (63) at (18, 1.5) {4};
		\node [style=none, scale=0.8] (64) at (19.5, 1.5) {5};
		\node [style=none, scale=0.8] (65) at (21, 1.5) {6};
		\node [style=none, scale=0.8] (66) at (22.5, 1.5) {7};
		\node [style=none, scale=0.8] (67) at (17, -0.25) {8};
		\node [style=none, scale=0.8] (68) at (18.5, -0.25) {10};
		\node [style=none, scale=0.8] (69) at (17, -1.5) {12};
		\node [style=none, scale=0.8] (70) at (18.5, -1.5) {11};
	\end{pgfonlayer}
	\begin{pgfonlayer}{edgelayer}
		\draw [dashed, line width=1pt](0) to (1);
		\draw [dashed, line width=1pt](1) to (2);
		\draw  [dashed, line width=1pt](2) to (3);
		\draw  [dashed, line width=1pt](3) to (4);
		\draw  [dashed, line width=1pt](4) to (5);
		\draw  [dashed, line width=1pt](5) to (6);
		\draw (2) to (7);
		\draw (3) to (9);
		\draw (9) to (11);
		\draw  [dashed, line width=1pt](26) to (27);
		\draw  [dashed, line width=1pt](27) to (28);
		\draw  [dashed, line width=1pt](28) to (29);
		\draw  [dashed, line width=1pt](29) to (30);
		\draw  [dashed, line width=1pt](26) to (31);
		\draw  [dashed, line width=1pt](31) to (32);
		\draw (27) to (33);
		\draw (33) to (35);
		\draw  [dashed, line width=1pt](50) to (51);
		\draw  [dashed, line width=1pt](51) to (52);
		\draw  [dashed, line width=1pt](52) to (53);
		\draw  [dashed, line width=1pt](53) to (54);
		\draw  [dashed, line width=1pt](50) to (55);
		\draw  [dashed, line width=1pt](55) to (58);
		\draw (51) to (57);
		\draw (57) to (59);
	\end{pgfonlayer}
\end{tikzpicture}

\caption{Subtrees induced by different diametrical paths (dashed) of the tree in \Cref{fig:A tree}.}
\label{fig:3subtreesofT}
\end{figure}
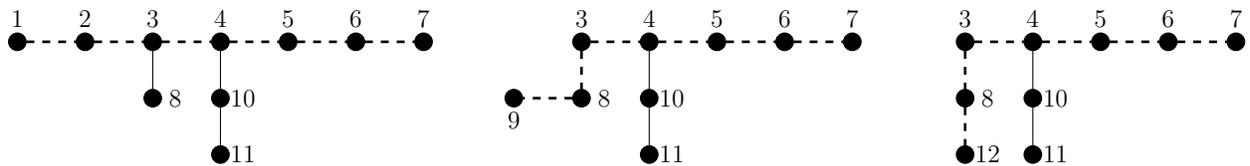

Note that the structure of the eccentric graph of a subtree induced from a diametrical path in $T$ depends on the diameter of $T$. In case of an even diameter, it looks as shown in the left of \Cref{fig:2typeOfDiamIndTrees} and in case of odd diameter, it looks as shown in the right of \Cref{fig:2typeOfDiamIndTrees}.

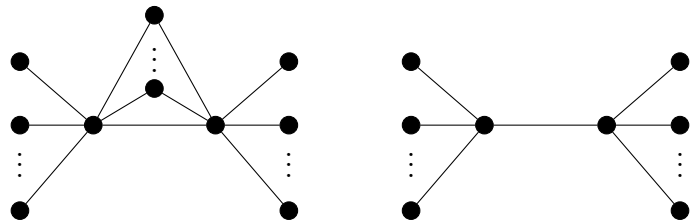
\begin{figure}[H]
\centering
\begin{tikzpicture}[x=0.65 cm,y=0.65 cm]
	\begin{pgfonlayer}{nodelayer}
		\node [style=doty, scale=0.6] (0) at (-5.5, 1) {};
		\node [style=doty, scale=0.6] (1) at (-4, 1) {};
		\node [style=doty, scale=0.6] (2) at (-1.5, 1) {};
		\node [style=doty, scale=0.6] (3) at (0, 1) {};
		\node [style=doty, scale=0.6] (4) at (-5.5, 2.3) {};
		\node [style=doty, scale=0.6] (5) at (0, 2.3) {};
		\node [style=doty, scale=0.6] (6) at (-5.5, -0.75) {};
		\node [style=doty, scale=0.6] (7) at (0, -0.75) {};
		\node [style=doty, scale=0.6] (8) at (2.5, 1) {};
		\node [style=doty, scale=0.6] (9) at (4, 1) {};
		\node [style=doty, scale=0.6] (10) at (6.5, 1) {};
		\node [style=doty, scale=0.6] (11) at (8, 1) {};
		\node [style=doty, scale=0.6] (12) at (2.5, 2.3) {};
		\node [style=doty, scale=0.6] (13) at (8, 2.3) {};
		\node [style=doty, scale=0.6] (14) at (2.5, -0.75) {};
		\node [style=doty, scale=0.6] (15) at (8, -0.75) {};
		\node [style=doty, scale=0.6] (16) at (-2.75, 1.75) {};
		\node [style=doty, scale=0.6] (17) at (-2.75, 3.25) {};
		\node [style=none] (18) at (-2.75, 2.5) {};
		\node [style=none] (19) at (-2.75, 2.5) {$\vdots$};
		\node [style=none] (20) at (-5.5, 0.35) {$\vdots$};
		\node [style=none] (21) at (0, 0.35) {$\vdots$};
		\node [style=none] (22) at (2.5, 0.35) {$\vdots$};
		\node [style=none] (23) at (8, 0.35) {$\vdots$};
	\end{pgfonlayer}
	\begin{pgfonlayer}{edgelayer}
		\draw (0) to (1);
		\draw (1) to (2);
		\draw (2) to (3);
		\draw (4) to (1);
		\draw (5) to (2);
		\draw (1) to (6);
		\draw (2) to (7);
		\draw (8) to (9);
		\draw (9) to (10);
		\draw (10) to (11);
		\draw (12) to (9);
		\draw (13) to (10);
		\draw (9) to (14);
		\draw (10) to (15);
		\draw (17) to (1);
		\draw (17) to (2);
		\draw (16) to (1);
		\draw (16) to (2);
	\end{pgfonlayer}
\end{tikzpicture}

\caption{Eccentric graph of subtrees induced by diametrical paths.}
\label{fig:2typeOfDiamIndTrees}
\end{figure}

{For example, the eccentric graphs of the subtrees in \Cref{fig:3subtreesofT} have been shown in \Cref{fig:EccGraphof3subtreesofT}.}

\begin{figure}[H]
    \centering
    \begin{tikzpicture}[x=0.6 cm, y=0.6 cm]
	\begin{pgfonlayer}{nodelayer}
		\node [style=doty, scale=0.6] (0) at (-7.5, 2) {};
		\node [style=doty, scale=0.6] (1) at (-4.5, 2) {};
		\node [style=doty, scale=0.6] (2) at (-6, 3) {};
		\node [style=doty, scale=0.6] (3) at (-6, 4) {};
		\node [style=doty, scale=0.6] (4) at (-6, 5.25) {};
		\node [style=doty, scale=0.6] (5) at (-9, 3) {};
		\node [style=doty, scale=0.6] (6) at (-9, 1) {};
		\node [style=doty, scale=0.6] (7) at (-3, 3) {};
		\node [style=doty, scale=0.6] (8) at (-3, 1) {};
		\node [style=none] (9) at (-7.5, 1.5) {1};
		\node [style=none] (10) at (-4.5, 1.5) {7};
		\node [style=none] (11) at (-6, 2.5) {};
		\node [style=none] (12) at (-6, 2.5) {};
		\node [style=none] (13) at (-6, 2.5) {4};
		\node [style=none] (14) at (-6, 3.5) {10};
		\node [style=none] (15) at (-6, 4.75) {11};
		\node [style=none] (16) at (-9.5, 3) {5};
		\node [style=none] (17) at (-9.5, 1) {6};
		\node [style=none] (18) at (-2.5, 3) {2};
		\node [style=none] (19) at (-2.5, 1) {8};
		\node [style=doty, scale=0.6] (20) at (0.75, 2) {};
		\node [style=doty, scale=0.6] (21) at (3.75, 2) {};
		\node [style=doty, scale=0.6] (22) at (2.25, 3) {};
		\node [style=doty, scale=0.6] (23) at (2.25, 4) {};
		\node [style=doty, scale=0.6] (24) at (2.25, 5.25) {};
		\node [style=doty, scale=0.6] (25) at (-0.75, 3) {};
		\node [style=doty, scale=0.6] (26) at (-0.75, 1) {};
		\node [style=doty, scale=0.6] (27) at (5.25, 3) {};
		\node [style=doty, scale=0.6] (28) at (5.25, 1) {};
		\node [style=none] (29) at (0.75, 1.5) {9};
		\node [style=none] (30) at (3.75, 1.5) {7};
		\node [style=none] (31) at (2.25, 2.5) {};
		\node [style=none] (32) at (2.25, 2.5) {};
		\node [style=none] (33) at (2.25, 2.5) {4};
		\node [style=none] (34) at (2.25, 3.5) {10};
		\node [style=none] (35) at (2.25, 4.75) {11};
		\node [style=none] (36) at (-1.25, 3) {5};
		\node [style=none] (37) at (-1.25, 1) {6};
		\node [style=none] (38) at (5.75, 3) {3};
		\node [style=none] (39) at (5.75, 1) {8};
		\node [style=doty, scale=0.6] (40) at (9, 2) {};
		\node [style=doty, scale=0.6] (41) at (12, 2) {};
		\node [style=doty, scale=0.6] (42) at (10.5, 3) {};
		\node [style=doty, scale=0.6] (43) at (10.5, 4) {};
		\node [style=doty, scale=0.6] (44) at (10.5, 5.25) {};
		\node [style=doty, scale=0.6] (45) at (7.5, 3) {};
		\node [style=doty, scale=0.6] (46) at (7.5, 1) {};
		\node [style=doty, scale=0.6] (47) at (13.5, 3) {};
		\node [style=doty, scale=0.6] (48) at (13.5, 1) {};
		\node [style=none] (49) at (9, 1.5) {12};
		\node [style=none] (50) at (12, 1.5) {7};
		\node [style=none] (51) at (10.5, 2.5) {};
		\node [style=none] (52) at (10.5, 2.5) {};
		\node [style=none] (53) at (10.5, 2.5) {4};
		\node [style=none] (54) at (10.5, 3.5) {10};
		\node [style=none] (55) at (10.5, 4.75) {11};
		\node [style=none] (56) at (7, 3) {5};
		\node [style=none] (57) at (7, 1) {6};
		\node [style=none] (58) at (14, 3) {3};
		\node [style=none] (59) at (14, 1) {8};
		\node [style=doty, scale=0.6] (60) at (-3, 2) {};
		\node [style=none] (61) at (-2.5, 2) {3};
	\end{pgfonlayer}
	\begin{pgfonlayer}{edgelayer}
		\draw (0) to (1);
		\draw (1) to (7);
		\draw (1) to (8);
		\draw (2) to (1);
		\draw (2) to (0);
		\draw (3) to (1);
		\draw (3) to (0);
		\draw (4) to (1);
		\draw (4) to (0);
		\draw (5) to (0);
		\draw (0) to (6);
		\draw (20) to (21);
		\draw (21) to (27);
		\draw (21) to (28);
		\draw (22) to (21);
		\draw (22) to (20);
		\draw (23) to (21);
		\draw (23) to (20);
		\draw (24) to (21);
		\draw (24) to (20);
		\draw (25) to (20);
		\draw (20) to (26);
		\draw (40) to (41);
		\draw (41) to (47);
		\draw (41) to (48);
		\draw (42) to (41);
		\draw (42) to (40);
		\draw (43) to (41);
		\draw (43) to (40);
		\draw (44) to (41);
		\draw (44) to (40);
		\draw (45) to (40);
		\draw (40) to (46);
		\draw (1) to (60);
	\end{pgfonlayer}
\end{tikzpicture}

    \caption{Eccentric graphs of the three subtrees in \Cref{fig:3subtreesofT} of the tree in \Cref{fig:A tree}.}
    \label{fig:EccGraphof3subtreesofT}
\end{figure}
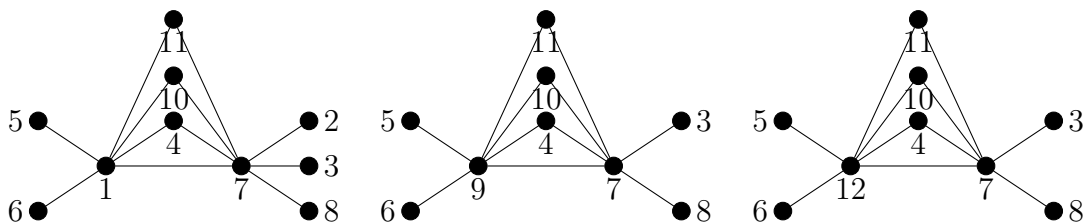
The following result shows that the graphs shown in \Cref{fig:2typeOfDiamIndTrees} are the building blocks for the eccentric graph of a tree.
\begin{theorem}
\label{thm:structureThrm}
    Let $Q_1,\cdots,Q_k$ be possible diametrical paths in $T$ with starting point $v_0^1,\dots,v_0^k$ and ending point $v_n^1,\dots,v_n^k$, respectively. Let $T_1,\dots,T_k$ be induced trees from $Q_1,\dots,Q_k$, repectively. Then, $E(T)=\cup_{i=1}^k E(T_k)$.
\end{theorem}
\begin{proof}
    It is clear that each vertex of $T$ lies in at least one tree induced from a diametrical path. 

For $i\in [k]$, let $e$ be an edge in the eccentric graph $E(T_i)$. As $Q_i$ is the unique diametrical path in $E(T_i)$, it follows that one of the endpoints of $e$ is either $v_0^i$ or $v_n^i$, assume $e=(v,v_n^i)$. Thus, $e_{T_i}(v)=d_{T_i}(v,v_n^i)=d_{T}(v,v_n^i)=e_{T}(v)$. Thus, $E(T_i)$ is a subgraph of $E(T)$.

Now, let $v\sim_{E(T)} w$; if both $v$ and $w$ lie on the same diametrical path, then done. Otherwise, it is enough to show that $v$ and $w$ both lie on the same tree $T_s$ for some $s\in [k]$. Let $v\in E(T_j)$, eccentric graph induced from the path $Q_j$ for some $j\in [k]$. If $w\notin E(T_j)$, then $w$ lies on a stem at some leaf $z$ in $T$. In that case, either the path joining from $v_0^j$ to $z$ or the path joining from $v_n^j$ to $z$ is a diametrical path. Consequently, either vertex $v$ lie on the tree induced by this diametrical path or both the vertices $v$ and $w$ lie on another diametrical path. In both cases, we get the adjacency relation between $v$ and $w$ in the eccentric graph of a tree induced by some diametrical path.
\end{proof}

The following example illustrates \Cref{thm:structureThrm}.
\begin{example}

\label{exm:structureThm}
    Let $T$ be the tree shown in \Cref{fig:A tree}. The eccentric graph of $T$ (see \Cref{fig:ExmplofStrThm}) is the union of the eccentric graphs (shown in \Cref{fig:EccGraphof3subtreesofT}) of the subtrees (shown in \Cref{fig:3subtreesofT}) induced from the three diametrical paths of $T$.
 \end{example}
\begin{figure}[H]
    \centering
   \begin{tikzpicture}[x=0.6 cm,y=0.6 cm]
	\begin{pgfonlayer}{nodelayer}
		\node [style=doty, scale=0.6] (0) at (-3.75, 4) {};
		\node [style=doty, scale=0.6] (1) at (2, 2) {};
		\node [style=doty, scale=0.6] (2) at (4.5, 3.5) {};
		\node [style=doty, scale=0.6] (3) at (4.5, 2.25) {};
		\node [style=doty, scale=0.6] (4) at (4.5, 1) {};
		\node [style=doty, scale=0.6] (5) at (-6.25, 4.25) {};
		\node [style=doty, scale=0.6] (6) at (-6.5, 2) {};
		\node [style=doty, scale=0.6] (7) at (0.25, 4.5) {};
		\node [style=doty, scale=0.6] (8) at (0.75, 5.75) {};
		\node [style=doty, scale=0.6] (9) at (1.25, 7.25) {};
		\node [style=doty, scale=0.6] (10) at (-3.75, 2.5) {};
		\node [style=doty, scale=0.6] (11) at (-3.75, 0.75) {};
		\node [style=none] (12) at (0.22, 3.85) {4};
		\node [style=none] (13) at (0.7, 5.1) {10};
		\node [style=none] (14) at (1.7, 7.32) {11};
		\node [style=none] (15) at (5, 3.5) {2};
		\node [style=none] (16) at (5, 2.25) {3};
		\node [style=none] (17) at (5, 1) {8};
		\node [style=none] (18) at (2, 1.5) {7};
		\node [style=none] (19) at (-6.75, 4.5) {5};
		\node [style=none] (20) at (-7.25, 1.75) {6};
		\node [style=none] (21) at (-3.75, 3.5) {9};
		\node [style=none] (22) at (-3.75, 2) {1};
		\node [style=none] (23) at (-3.75, 0.25) {12};
	\end{pgfonlayer}
	\begin{pgfonlayer}{edgelayer}
		\draw (0) to (1);
		\draw (1) to (2);
		\draw (1) to (3);
		\draw (1) to (4);
		\draw (1) to (7);
		\draw (7) to (0);
		\draw (8) to (1);
		\draw (8) to (0);
		\draw (9) to (0);
		\draw (9) to (1);
		\draw (5) to (0);
		\draw (0) to (6);
		\draw (1) to (10);
		\draw (7) to (10);
		\draw (8) to (10);
		\draw (9) to (10);
		\draw (5) to (10);
		\draw (6) to (10);
		\draw (5) to (11);
		\draw (6) to (11);
		\draw (11) to (1);
		\draw (7) to (11);
		\draw (8) to (11);
		\draw (9) to (11);
	\end{pgfonlayer}
\end{tikzpicture}
    \caption{Eccentric graph of the tree in \Cref{fig:A tree} which is the union of the graphs in \Cref{fig:EccGraphof3subtreesofT}.}
    \label{fig:ExmplofStrThm}
\end{figure}
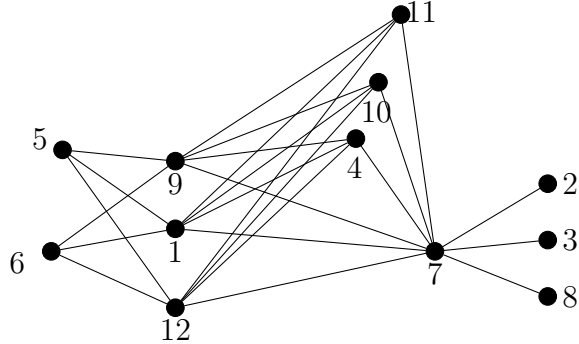
In the remaining part of this section, we will highlight more structural information about the eccentric graph of a tree.

\bp
\label{prop:Ecc_is_smallest_or_largest_in_tree}
Let $T$ be a tree. There does not exist $v_1,v_2,v_3\in V(T)$ such that $v_1\sim_{E(T)} v_2$, $v_2\sim_{E(T)} v_3$ and $e_T(v_1)<e_T(v_2)<e_T(v_3)$.
\ep
\NI \pf
On the contrary, assume such $v_1,v_2,v_3\in V(T)$ exist. Then $d_T(v_1,v_2)=e_T(v_1)$, $d_T(v_2,v_3)=e_T(v_2)$, and $v_2$ and $v_3$ are pendant vertices in $T$. Let $P$ be the path between $v_1$ and $v_2$ in $T$ and $v_k$ be the middle vertex on $P$ if length of $P$ is even else $v_k$ be the right middle vertex on $P$. Now as $v_2$ and $v_3$
 are pendants, $v_3$ cannot lie on the path $P$, moreover $v_3$ will lie on some branch emerging from a vertex $w$ on the path $P$ other than $v_2$. Now two cases arise depending on whether $w$ is positioned to the right of $v_k$ or lies strictly left to $v_k$.
  \begin{figure}[H]
 \centering
    \begin{tikzpicture}[x=0.65cm, y=0.65cm]
	\begin{pgfonlayer}{nodelayer}
		\node [style=doty, scale=0.6] (0) at (-9, 1) {};
		\node [style=doty, scale=0.6] (1) at (-8, 1) {};
		\node [style=doty, scale=0.6] (2) at (-6, 1) {};
		\node [style=doty, scale=0.6] (3) at (-4, 1) {};
		\node [style=doty, scale=0.6] (4) at (-3, 1) {};
		\node [style=doty, scale=0.6] (5) at (-0.5, 1) {};
		\node [style=doty, scale=0.6] (6) at (0.5, 1) {};
		\node [style=doty, scale=0.6] (7) at (2.5, 1) {};
		\node [style=doty, scale=0.6] (8) at (4.5, 1) {};
		\node [style=doty, scale=0.6] (9) at (5.5, 1) {};
		\node [style=doty, scale=0.6] (10) at (-5, 1) {};
		\node [style=doty, scale=0.6] (11) at (1.5, 1) {};
		\node [style=doty, scale=0.6] (12) at (-3.5, 3.5) {};Figure
		\node [style=doty, scale=0.6] (13) at (0, 3.5) {};
		\node [style=none] (14) at (-9, 0.55) {$v_1$};
		\node [style=none] (15) at (-6, 0.55) {$v_k$};
		\node [style=none] (16) at (-5, 0.55) {$w$};
		\node [style=none] (17) at (-3, 0.55) {$v_2$};
		\node [style=none] (18) at (-0.5, 0.55) {$v_1$};
		\node [style=none] (19) at (1.5, 0.55) {$w$};
		\node [style=none] (20) at (2.5, 0.55) {$v_k$};
		\node [style=none] (21) at (5.5, 0.55) {$v_2$};
		\node [style=none] (22) at (-3.5, 3.9) {$v_3$};
		\node [style=none] (23) at (0, 3.9) {$v_3$};
		\node [style=none] (26) at (-6, -0.2) {case $1$};
		\node [style=none] (27) at (2.25, -0.2) {case $2$};
	\end{pgfonlayer}
	\begin{pgfonlayer}{edgelayer}
		\draw [line width=1pt](0) to (1);
		\draw [line width=1pt, dashed](1) to (2);
		\draw [line width=1pt,dashed](2) to (3);
		\draw [line width=1pt](3) to (4);
		\draw [line width=1pt](5) to (6);
		\draw [line width=1pt,dashed](6) to (7);
		\draw [line width=1pt,dashed](7) to (8);
		\draw [line width=1pt](8) to (9);
		\draw [line width=1.2pt, in=90, out=-90, looseness=1.25,dotted] (12) to (10);
		\draw [line width=1.2pt, in=90, out=-90, dotted] (13) to (11);
	\end{pgfonlayer}
\end{tikzpicture}
     \caption{}
     \label{fig:lemm7-1}
 \end{figure}

If $w$ lies on the right to $v_k$ as shown in the left of \Cref{fig:lemm7-1}, then $d_T(w, v_2)\leq d_T(w,v_1)$. 
As $d_T(v_2,v_1)=e_T(v_1)<e_T(v_2)=d_T(v_2,v_3)$, then $d_T(w,v_1)<d_T(w,v_3)$.
\newline Thus,
\begin{align*}
    e_T(v_1)~=~d_T(v_1,v_2)&=d_T(v_1,w)+d_T(w,v_2),\\
    &\leq  d_T(v_1,w)+d_T(v_1,w),\\
    &< d_T(v_1,w)+d_T(w,v_3),\\
    &= d_T(v_1,v_3),
\end{align*}
which is a contradiction.

 If $w$ lies strictly left to $v_k$ as shown in the right of \Cref{fig:lemm7-1}. Since $e_T(v_3)>e_T(v_2)$, there exist a vertex, $w_1$ in $V(T)$ such that  $e_T(v_3)=d_T(v_3,w_1)$. In particular, \begin{equation}
 \label{eqn:case2}
     d_T(v_3,v_2)<d_T(v_3,w_1).
 \end{equation} Note that $w_1$ must lie on some branch emerging from a vertex on $P$ else eccentricity of $v_2$ will increase. This leads to the following two subcases (see \Cref{fig:lemm7-2}):
 \begin{figure}[H]
\centering
   \begin{tikzpicture}[x=0.6cm, y=0.6cm]
	\begin{pgfonlayer}{nodelayer}
		\node [style=doty, scale=0.6] (5) at (-0.5, 1) {};
		\node [style=doty, scale=0.6] (7) at (2.5, 1) {};
		\node [style=doty, scale=0.6] (8) at (4.5, 1) {};
		\node [style=doty, scale=0.6] (9) at (5.5, 1) {};
		\node [style=doty, scale=0.6] (11) at (0.5, 1) {};
		\node [style=doty, scale=0.6] (13) at (-1, 3.5) {};
		\node [style=none] (18) at (-1.5, 0.55) {$v_1$};
		\node [style=none] (19) at (0.5, 0.55) {$w$};
		\node [style=none] (20) at (2.5, 0.55) {$v_k$};
		\node [style=none] (21) at (5.5, 0.55) {$v_2$};
		\node [style=none] (23) at (-1, 4) {$v_3$};
		\node [style=none] (27) at (2.5, -3.55) {case $2(ii)$};
		\node [style=doty, scale=0.6] (28) at (-9, 1) {};
        \node [style=doty, scale=0.6] (49) at (-10, 1) {};
        \node [style=doty, scale=0.6] (50) at (-1.5, 1) {};
		\node [style=doty, scale=0.6] (29) at (-8, 1) {};
		\node [style=doty, scale=0.6] (30) at (-6, 1) {};
		\node [style=doty, scale=0.6] (31) at (-4, 1) {};
		\node [style=doty, scale=0.6] (32) at (-3, 1) {};
		\node [style=doty, scale=0.6] (33) at (-7, 1) {};
		\node [style=doty, scale=0.6] (34) at (-8.5, 3.5) {};
		\node [style=none] (35) at (-10, 0.55) {$v_1$};
		\node [style=none] (36) at (-7, 0.55) {$w$};
		\node [style=none] (37) at (-6, 0.55) {$v_k$};
		\node [style=none] (38) at (-3, 0.55) {$v_2$};
		\node [style=none] (39) at (-8.5, 4) {$v_3$};
		\node [style=none] (41) at (-6.25, -3) {case $2(i)$};
		\node [style=doty, scale=0.6] (42) at (2.25, -1.5) {};
		\node [style=doty, scale=0.6] (43) at (-8.5, -1.5) {};
		\node [style=doty, scale=0.6] (44) at (1.5, 1) {};
		\node [style=none] (45) at (-8, 1.35) {$w_2$};
		\node [style=none] (46) at (1.5, 1.35) {$w_2$};
		\node [style=none] (47) at (-8.5, -2) {$w_1$};
		\node [style=none] (48) at (2.25, -2) {$w_1$};
	\end{pgfonlayer}
	\begin{pgfonlayer}{edgelayer}
		\draw [line width=1pt,dashed](7) to (8);
		\draw [line width=1pt](8) to (9);
		\draw [line width=1.2pt, in=90, out=-90,dotted] (13) to (11);
		\draw [line width=1pt,dashed](28) to (29);
  \draw [line width=1pt](28) to (49);
  \draw [line width=1pt](5) to (50);
		\draw [line width=1pt,dashed](29) to (30);
		\draw [line width=1pt,dashed](30) to (31);
		\draw [line width=1pt](31) to (32);
		\draw [line width=1.2pt, in=90, out=-90,dotted] (34) to (33);
		\draw [line width=1pt, in=45, out=-135, looseness=1.25,red,dashed] (29) to (43);
		\draw [line width=1pt,dashed](5) to (7);
		\draw [line width=1pt, in=135, out=-30, looseness=1.25,red,dashed] (44) to (42);
	\end{pgfonlayer}
\end{tikzpicture}
    \caption{}
    \label{fig:lemm7-2}
\end{figure}

First, $w_1$ lies on a branch emerging from a vertex, $w_2$ which is situated to the left of $w$. By \eqref{eqn:case2}, $d_T(w,v_2)<d_T(w,w_1)$ and $d_T(w,v_2)\geq d_T(w,v_3)$ otherwise $d_T(v_1,v_2)=e_T(v_1)<d_T(v_1, v_3)$. Thus
\begin{align*}
    d_T(v_2,w_1)&=d_T(w,v_2)+d_T(w,w_1)\\
    &> d_T(w, v_3)+d_T(w,v_2),\\
    &= d_T(v_2,v_3),\\
    &=e_T(v_2)
\end{align*}
which is a contradiction.

Second, if $w_2$ is on the right to $w$ as shown in right of \Cref{fig:lemm7-2}. Again by \eqref{eqn:case2}, $d_T(w_2,v_2)<d_T(w_2,w_1)$. Hence, $e_T(v_1)=d_T(v_1,v_2)<d_T(v_1,w_1)$, which is absurd.

\qed

\emph{The essence of Proposition \ref{prop:Ecc_is_smallest_or_largest_in_tree} can be summarized as
the eccentricity of a vertex $v\in V(T)$ is either the smallest or the largest among the eccentricities of its neighbours in the eccentric graph of $T$.}
\section{Eccentric girth of a tree}
\label{sec:EccGir of a tree}

In this section, we will determine the eccentric girth of a tree and its potential values. In addition, we will classify the instances in which these possible values of the eccentric girth can be achieved. It is well-known that two paths of maximum length must pass through a common point. Thus, it is evident that two diametrical paths in a tree must intersect. But this  is not true in general, the graph in \Cref{fig:countExm} has two diametrical paths (dashed) but they do not intersect. 
\begin{figure}[ht]
    \centering
    \begin{tikzpicture}
	\begin{pgfonlayer}{nodelayer}
		\node [style=none] (7) at (-7, 3) {};
		\node [style=none] (8) at (-6, 3) {};
		\node [style=doty, scale=0.5] (9) at (-5, 3) {};
		\node [style=doty, scale=0.5] (10) at (-4, 3) {};
		\node [style=doty, scale=0.5] (11) at (-3, 3) {};
		\node [style=doty, scale=0.5] (12) at (-2, 3) {};
		\node [style=doty, scale=0.5] (14) at (-7.44, 3) {};
		\node [style=doty, scale=0.5] (15) at (-6, 3) {};
		\node [style=doty, scale=0.5] (16) at (-7.44, 4.5) {};
		\node [style=doty, scale=0.5] (17) at (-6, 4.5) {};
		\node [style=doty, scale=0.5] (18) at (-4, 4.5) {};
		\node [style=doty, scale=0.5] (19) at (-3, 4.5) {};
		\node [style=doty, scale=0.5] (20) at (-2, 4.5) {};
		\node [style=doty, scale=0.5] (21) at (-1, 4.5) {};
	\end{pgfonlayer}
	\begin{pgfonlayer}{edgelayer}
		\draw [line width=1pt,dashed](16) to (17);
		\draw [line width=1pt,dashed](17) to (18);
		\draw [line width=1pt,dashed](19) to (20);
		\draw [line width=1pt,dashed](14) to (15);
		\draw [line width=1pt,dashed](15) to (9);
		\draw [line width=1pt,dashed](10) to (11);
		\draw [line width=1pt,dashed](18) to (19);
		\draw [line width=1pt,dashed](20) to (21);
		\draw [line width=1pt,dashed](9) to (10);
		\draw [line width=1pt,dashed](11) to (12);
		\draw [line width=1pt](20) to (12);
		\draw [line width=1pt](19) to (11);
		\draw [line width=1pt](18) to (10);
		\draw [line width=1pt](17) to (15);
		\draw [line width=1pt](16) to (14);
		\draw [line width=1pt](16) to (15);
		\draw [line width=1pt](17) to (14);
	\end{pgfonlayer}
\end{tikzpicture}
\caption{ A graph having two non-intersecting diametrical paths}
    \label{fig:countExm}
\end{figure}
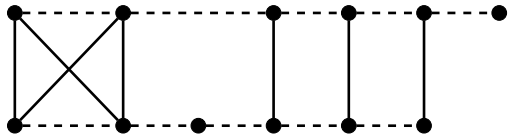

Now, we will present the main result of this section which classifies the eccentric girth of a tree.
\bt \label{Thm: Tree_Girth}
Let $T$ be a tree. Then the girth of the eccentric graph $E(T)$, is either zero, three, or four.
Moreover,
\[
\text{the girth of } E(T)=
\begin{cases}
    3&\text{ if the diameter of $T$ is even,}\\
    0&\text{ if the diameter of $T$ is odd with unique diametrical path,}\\
    4& \text{ otherwise}.
\end{cases}
\]
\et
\noindent\pf
The proof is divided into the following cases depending on the parity of the diameter of $T$.

First, let the diameter of $T$ is even and $P= v_0\,v_1\ldots v_k\,v_{k+1}\ldots v_{2k}$ be a diametrical path. Note that $e(v_0)=2k=e(v_{2k})$ and $d(v_0,v_{2k})=2k$, therefore $v_0\sim_{E(T)} v_{2k}$. 
If $e(v_k)>k$ then one of $e(v_0)$ or $e(v_{2k})$ will be greater than $2k$, which is not possible. Also, $d(v_0,v_k)=k=d(v_k,v_{2k})$, therefore $e(v_k)=k$ and $v_k\sim_{E(T)} v_0$, $v_k\sim_{E(T)} v_{2k}$. Thus, $v_0,v_k,$ and $v_{2k}$ forms a triangle in $E(T)$.

Second, If the diameter of $T$ is odd and $P= v_0\,v_1\ldots v_k\,v_{k+1}\ldots v_{2k+1}$ is the unique diametrical path in $T$. It is sufficient to show that for any vertex $i\in V(T)$ exactly one of $v_0$ or $v_{2k+1}$ is eccentric to $i$ and no other vertex is eccentric to $i$. Note that, in a tree, if a vertex $j$ is eccentric to some vertex then $j$ must be a pendant vertex. 

Let $i\in V(P)$, if possible, there exists a vertex $j\in V(T)$ other than $v_0$ and $v_{2k+1}$ which is eccentric to $i$, that is, $d(i, j)=e(i)$, then $j$ is a leaf of a branch emerging from some vertex $p\in V(P)$. Assume $p$ is on the left of $i$ in $P$ then $d(i,j)\geq d(i, v_0)$ this implies $d(v_{2k+1},j)=d(v_{2k+1},i)+d(i, j)\geq d(v_{2k+1},i)+d(i, v_0)=2k+1$ which contradicts the fact that $P$ is the only diametrical path. A similar argument can be given when $p$ is on the right of $i$. 

Now suppose $i\in V(T)\setminus V(P)$ lying on some branch emerging from a vertex $i'\in V(P)$. Again let there exists $j\in V(T)$ other than $v_0$ and $v_{2k+1}$ which is eccentric to $i$. Note that $j$ cannot lie on the same branch otherwise eccentricity of one of $v_0$ or $v_{2k+1}$ will increase. Thus, $j$ must be eccentric to $i'$ which cannot happen as proved in the preceding paragraph. 
Moreover, because of odd diameter exactly one of  $v_0$ or $ v_{2k+1}$ can be eccentric to $i$. 
For illustration, $E(T)$ in this scenario is shown in \Cref{fig:Ecc_Tree_Zero_Girth}.

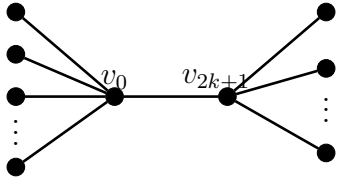
\begin{figure}[h!]
    \centering
   \begin{tikzpicture}[x=0.75 cm, y=0.75 cm]
	\begin{pgfonlayer}{nodelayer}
		\node [style=doty, scale=0.6] (0) at (-1, 1) {};
		\node [style=doty, scale=0.6] (1) at (1, 1) {};
		\node [style=doty, scale=0.6] (2) at (-2.75, 2.5) {};
		\node [style=doty, scale=0.6] (3) at (-2.75, 1.75) {};
		\node [style=doty, scale=0.6] (4) at (-2.75, 1) {};
		\node [style=doty, scale=0.6] (5) at (-2.75, -0.25) {};
		\node [style=doty, scale=0.6] (6) at (2.75, 2.5) {};
		\node [style=doty, scale=0.6] (7) at (2.75, 1.5) {};
		\node [style=doty, scale=0.6] (8) at (2.75, 0) {};
        \node [style=none] (9) at (-1, 1.3) {$v_0$};
        \node [style=none] (9) at (0.8, 1.3) {$v_{2k+1}$};
		\node [style=none] (9) at (-2.75, 0.5) {$\vdots$};
		\node [style=none] (10) at (2.75, 0.9) {$\vdots$};
	\end{pgfonlayer}
	\begin{pgfonlayer}{edgelayer}
		\draw [line width=1](0) to (1);
		\draw [line width=1](1) to (6);
		\draw [line width=1](7) to (1);
		\draw [line width=1](1) to (8);
		\draw [line width=1](0) to (2);
		\draw [line width=1](0) to (3);
		\draw [line width=1](4) to (0);
		\draw [line width=1](0) to (5);
	\end{pgfonlayer}
\end{tikzpicture}

    \caption{Eccentric graph of a tree (of odd diameter) with unique diametrical path}
    \label{fig:Ecc_Tree_Zero_Girth}
\end{figure}

Third, let the diameter of $T$ is odd and $P= v_0\,v_1\ldots v_k\,v_{k+1}\ldots v_{2k+1}$, $P^{\prime}= w_0\,w_1\ldots w_k\allowbreak \,w_{k+1}\ldots w_{2k+1}$ be two diametrical paths in $T$.
As mentioned at the start of \Cref{sec:EccGir of a tree}, they must intersect. Therefore,
it is reasonable to assume
that $P$ and $P'$ have one common endpoint say $v_0=w_0$, otherwise we can create two such diametrical paths.
Hence, $(v_0,v_{2k+1},v_1,w_{2k+1})$ forms a 4-cycle  in $E(T)$. Now, if there is a triangle $(z_1, z_2, z_3)$ in $E(T)$ and $e(z_1)\leq e(z_2)\leq e(z_3)$. Without loss of generality, assume $z_1$ is a vertex on some branch emerging from $w_p$, $1\leq p\leq k$ (Note that $z_1$ can be $w_0$). If $z$ is any vertex eccentric to $z_1$ then $z$ must be a vertex on some branch emerging from $w_i$ for some $k+1\leq i\leq 2k$, if not then $d(z, w_{2k+1})>2k+1$ which is the diameter. Now $z_2$ being eccentric to $z_1$, must lie on some branch emerging from $w_q$, $k+1\leq q \leq 2k$ (Note that $z_2$ can be $w_{2k+1}$). Again, as $z_3$ is eccentric to $z_2$, $z_3$ is a vertex on some branch emerging from $w_r$, $1\leq r \leq k$ but then $z_3$ cannot be eccentric to $z_1$. Hence $E(T)$ cannot have a triangle.
\qed
\section{Eccentric graph of the Cartesian product of graphs}
\label{sec:EccGofCProd}

In this section, We will examine some properties of the eccentric graph of the Cartesian product of general graphs { and calculate the girth of the Cartesian product of trees in \Cref{sec:EccGirthOfCPofTrees}.} We begin by recalling the definition of the \emph{Cartesian product}  and the \emph{Kronecker product} of two graphs.
\bd
\label{def:Cartesian product}
Let $G_1$ and $G_2$ be two simple connected graphs. The \emph{Cartesian product} of $G_1$ and $G_2$ denoted as $G_1\square  G_2$, is a graph with vertex set $V(G_1)\times  V(G_2)$, and two vertices $(u_1,u_2)$ and $(v_1,v_2)$ are adjacent if and only if either $u_1=v_1$ and $u_2 \sim_{G_2} v_2$ 
or $u_1\sim_{G_1}  v_1$ 
and $u_2=v_2$. 
\ed

The following equations follow directly from \Cref{def:Cartesian product}.
 \begin{equation}
\label{eqn:dis is sum of dis}
    d_{G_1\square  G_2}\big((u_1,u_2),(v_1,v_2)\big)=d_{G_1}(u_1,v_1)+d_{G_2}(u_2,v_2),
\end{equation}
and
\begin{equation}
\label{eqn:ecc is sum of ecc}
e_{G_1\square  G_2}\big((u_1,u_2)\big)=e_{G_1}(u_1)+e_{G_2}(u_2).    
\end{equation}

The above definition and the equations can be generalised to the Cartesian product of $k$ graphs $G_1,\dots, G_k$ denoted as $G_1\square
\cdots\square
G_k$.

\bd
\label{def: Kronecker product}
Let $G_1$ and $G_2$ be two simple connected graphs. The \emph{Kronecker product} of $G_1$ and $G_2$ denoted as $G_1\times G_2$, is a graph with vertex set $V(G_1)\times  V(G_2)$, and two vertices $(u_1,u_2)$ and $(v_1,v_2)$ are adjacent if and only if $u_1\sim_{G_1}  v_1$ and $u_2 \sim_{G_2} v_2$.

\ed

\bl \label{lem:eccentricity_in_cartesian_product}
Let $G_1, \ldots, G_k$ be simple connected graphs
and $G=G_1\square \cdots \square G_k$ {be their Cartesian product}. Let $u=(u_1, \ldots, u_k)$, $v=(v_1, \ldots, v_k)\in V(G)$ where $u_i, v_i \in V(G_i)$ for $i\in [k]$. Then, $v$ is eccentric to $u$ if and only if $v_i$ is eccentric to $u_i$ for all $i\in [k]$.
\el
\NI\pf
Let $v$ \text{ is eccentric to } $u$, i.e. $d_G\big( u, v\big)= \max \{d_G\big(u, x\big): x\in V(G)\}$. Then by \eqref{eqn:dis is sum of dis} we can express this as:
\[
\sum\limits_{i=1}^{k} d_{G_i}\big( u_i, v_i\big)=\max \left\{ \sum\limits_{i=1}^{k} d_{G_i}\big(u_i, x_i\big): x_i \in V(G_i)\right\}.
\]
Which holds only if 
\[
d_{G_i}\big( u_i, v_i\big)=\max \{ d_{G_i}\big(u_i, x_i\big): x_i \in V(G_i)\} \text{ for all } i\in [k]\}.
\]
Thus, 
$v_i$ is eccentric to $u_i$ for all $i\in [k]$. Furthermore, we can reverse the steps of this argument to establish the converse part.
\qed

\emph{Note that
if $(u_1,\ldots, u_k)\sim_{E(G_1\square\cdots \square G_k)} (v_1,\ldots , v_k)$ then $u_i \neq v_i$ for all $i\in [k]$. Also, it is clear from \Cref{lem:eccentricity_in_cartesian_product} that if $u \sim_{E(G)} v$ then $u_i \sim_{E(G_i)} v_i$ for all $i\in [k]$ but the converse is not true.} {For example, $1\sim_{E(P_4)} 3$ and $2\sim_{E(P_4)} 4$ but $(1, 2)\nsim_{E(P_4 \square P_4)} (3,4)$ (see \Cref{fig:eccIsNotAdja}).
}

\begin{figure}[h!]
    \centering
 \begin{tikzpicture}[x=0.75 cm, y=0.75 cm]
	\begin{pgfonlayer}{nodelayer}
		\node [style=doty, scale=0.6] (0) at (-6, 14) {};
		\node [style=doty, scale=0.6] (1) at (-4.75, 14) {};
		\node [style=doty, scale=0.6] (2) at (-3.5, 14) {};
		\node [style=doty, scale=0.6] (3) at (-2.25, 14) {};
		\node [style=doty, scale=0.6] (10) at (1.5, 13) {};
		\node [style=doty, scale=0.6] (11) at (2.75, 13) {};
		\node [style=doty, scale=0.6] (12) at (4, 13) {};
		\node [style=doty, scale=0.6] (13) at (5.25, 13) {};
		\node [style=none] (14) at (-6, 14.5) {3};
		\node [style=none] (15) at (-4.75, 14.5) {1};
		\node [style=none] (16) at (-3.5, 14.5) {4};
		\node [style=none] (17) at (-2.25, 14.5) {2};
		\node [style=none] (24) at (0.75, 13) {(3,3)};
		\node [style=none] (25) at (2.75, 12.5) {(1,1)};
		\node [style=none] (26) at (4, 12.5) {(4,4)};
		\node [style=none] (27) at (6, 13) {(2,2)};
		\node [style=none] (28) at (-4.5, 11) {$E(P_4)$};
		\node [style=none] (30) at (3.5, 11) {$E(P_4\square P_4)$};
		\node [style=doty, scale=0.6] (31) at (1.5, 14) {};
		\node [style=doty, scale=0.6] (32) at (1.5, 12) {};
		\node [style=doty, scale=0.6] (33) at (5.25, 14) {};
		\node [style=doty, scale=0.6] (34) at (5.25, 12) {};
		\node [style=none] (35) at (0.75, 14) {(3,4)};
		\node [style=none] (36) at (0.75, 12) {(4,3)};
		\node [style=none] (37) at (6, 14) {(1,2)};
		\node [style=none] (38) at (6, 12) {(2,1)};
		\node [style=doty, scale=0.6] (39) at (1.5, 16.25) {};
		\node [style=doty, scale=0.6] (40) at (2.75, 16.25) {};
		\node [style=doty, scale=0.6] (41) at (4, 16.25) {};
		\node [style=doty, scale=0.6] (42) at (5.25, 16.25) {};
		\node [style=none] (43) at (0.75, 16.25) {(3,2)};
		\node [style=none] (44) at (2.75, 15.75) {(1,4)};
		\node [style=none] (45) at (4, 15.75) {(4,1)};
		\node [style=none] (46) at (6, 16.25) {(2,3)};
		\node [style=doty, scale=0.6] (47) at (1.5, 17.25) {};
		\node [style=doty, scale=0.6] (48) at (1.5, 15.25) {};
		\node [style=doty, scale=0.6] (49) at (5.25, 17.25) {};
		\node [style=doty, scale=0.6] (50) at (5.25, 15.25) {};
		\node [style=none] (51) at (0.75, 17.25) {(3,1)};
		\node [style=none] (52) at (0.75, 15.25) {(4,2)};
		\node [style=none] (53) at (6, 17.25) {(1,3)};
		\node [style=none] (54) at (6, 15.25) {(2,4)};
	\end{pgfonlayer}
	\begin{pgfonlayer}{edgelayer}
		\draw (0) to (1);
		\draw (1) to (2);
		\draw (2) to (3);
		\draw (10) to (11);
		\draw (11) to (12);
		\draw (12) to (13);
		\draw (31) to (11);
		\draw (11) to (32);
		\draw (12) to (33);
		\draw (12) to (34);
		\draw (39) to (40);
		\draw (40) to (41);
		\draw (41) to (42);
		\draw (47) to (40);
		\draw (40) to (48);
		\draw (41) to (49);
		\draw (41) to (50);
	\end{pgfonlayer}
\end{tikzpicture}

    \caption{Eccentric graph of naturally labelled $P_4$ and $P_4\square P_4$.}
    \label{fig:eccIsNotAdja}
\end{figure}
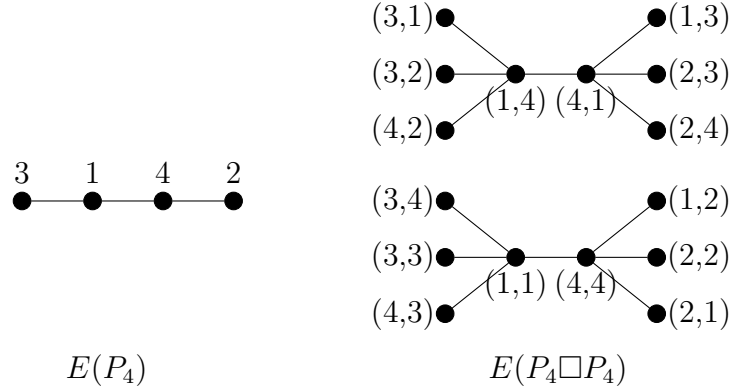

\begin{corollary}\label{cor: sameEccentricityProduct}
    Let $G_1$ and $G_2$ be simple graphs such that all the vertices in both $G_1$ and $G_2$ have the same eccentricities. Then $E(G_1\square G_2)$ is isomorphic to $E(G_1)\times E(G_2)$, the Kronecker product of $E(G_1)$ and $E(G_2)$.
\end{corollary}

\bl \label{lem:4Cycle_in_E(G)}
Let $G_1, \ldots, G_k$ be simple connected graphs and $G=G_1\square \cdots \square G_k$. If for some $s, t\in [k]$ there exist $u_s, v_s, w_s\in V(G_s)$ such that $u_s\sim_{E(G_s)} v_s$, $v_s\sim_{E(G_s)} w_s$ and $e_{G_s}(v_s)\geq \max\{e_{G_s}(u_s),e_{G_s}(w_s)\}$, and there exist $u_t, v_t, w_t\in V(G_t)$ such that $u_t\sim_{E(G_t)} v_t$, $v_t\sim_{E(G_t)} w_t$ and $e_{G_t}(v_t)\leq \min\{e_{G_t}(u_t),e_{G_t}(w_t)\}$, then there exists a $4$-cycle in $E(G)$.
 \el

 \NI\pf
 Without loss of generality, assume $s=1$ and $t=2$
 and for $i=3, \ldots, k$, let $\{u_i, v_i\}$ be an edge in $E(G_i)$ such that $e_{G_i}(u_i)\geq e_{G_i}(v_i)$.
Using \Cref{lem:eccentricity_in_cartesian_product}, $a=(u_1, v_2, v_3, \ldots v_k)$, $ b=(v_1, w_2, u_3, \ldots,\allowbreak u_k)$, $ c=(w_1, v_2, v_3, \ldots, v_k)$ and $ d=(v_1, u_2,u_3, \ldots, u_k)$ forms a $4$-cycle in $E(G)$.
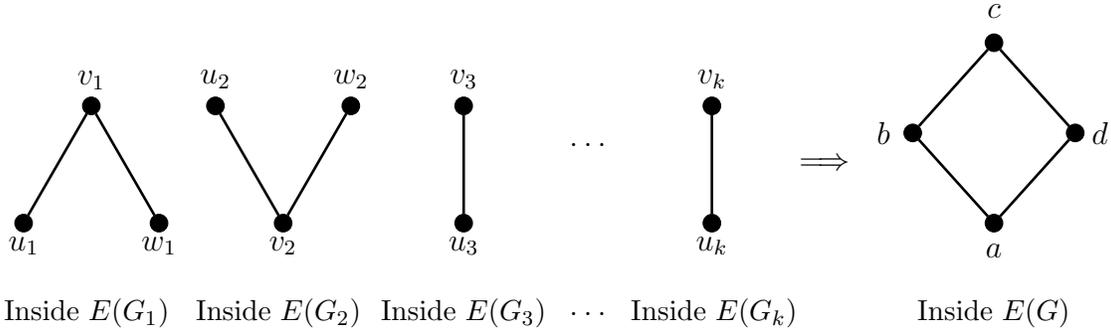
 \begin{figure}[H]
\centering
 \begin{tikzpicture}[x=0.6 cm, y=0.6 cm]
	\begin{pgfonlayer}{nodelayer}
		\node [style=doty, scale=0.6] (0) at (1, 2.6) {};
		\node [style=doty, scale=0.6] (1) at (-0.5, 0) {};
		\node [style=doty, scale=0.6] (2) at (2.5, 0) {};
		\node [style=doty, scale=0.6] (3) at (5.25, 0) {};
		\node [style=doty, scale=0.6] (4) at (3.75, 2.6) {};
		\node [style=doty, scale=0.6] (5) at (6.75, 2.6) {};
		\node [style=doty, scale=0.6] (6) at (19.2, 2) {};
		\node [style=doty, scale=0.6] (9) at (21, 4) {};
		\node [style=doty, scale=0.6] (7) at (21, 0) {};
		\node [style=doty, scale=0.6] (8) at (22.8, 2) {};
		\node [style=none] (23) at (-0.5, -0.5) {$u_1$};
		\node [style=none] (24) at (2.5, -0.5) {$w_1$};
		\node [style=none] (25) at (1, 3.17) {$v_1$};
		\node [style=none] (26) at (5.25, -0.5) {$v_2$};
		\node [style=none] (27) at (3.75, 3.2) {$u_2$};
		\node [style=none] (28) at (6.75, 3.2) {$w_2$};
		\node [style=none] (29) at (18.55, 2) {$b$};
		\node [style=none] (39) at (21, 4.7) {$c$};
		\node [style=none] (30) at (21, -0.6) {$a$};
		\node [style=none] (31) at (23.35, 2) {$d$};
		\node [style=none] (32) at (0.9, -2) {\small Inside $E(G_1)$};
		\node [style=none] (33) at (5.15, -2) {\small Inside $E(G_2)$};
		\node [style=none] (35) at (21, -2) {\small Inside $E(G)$};
		\node [style=none] (28) at (9.25, 3.2) {$v_3$};
		\node [style=none] (28) at (14.75, 3.2) {$v_k$};
		\node [style=none] (24) at (9.25, -0.5) {$u_3$};
		\node [style=none] (24) at (14.75, -0.5) {$u_k$};
		\node [style=none] (37) at (17.25, 1.3) {$\implies$};
		\node [style=doty, scale=0.6] (40) at (9.25, 2.6) {};
		\node [style=doty, scale=0.6] (41) at (9.25, 0) {};
		\node [style=doty, scale=0.6] (42) at (14.75, 2.6) {};
		\node [style=doty, scale=0.6] (43) at (14.75, 0) {};
		\node [style=none] (44) at (12, 1.75) {$\ldots$};
		\node [style=none] (45) at (9.25, -2) {\small Inside $E(G_3)$};
		\node [style=none] (46) at (14.8, -2) {\small Inside $E(G_k)$};
		\node [style=none] (47) at (12, -2) {$\ldots$};
	\end{pgfonlayer}
	\begin{pgfonlayer}{edgelayer}
		\draw [line width=1pt] (0) to (2);
		\draw [line width=1pt] (1) to (0);
		\draw [line width=1pt] (3) to (5);
		\draw [line width=1pt] (4) to (3);
		\draw [line width=1pt] (6) to (9);
		\draw [line width=1pt] (8) to (9);
		\draw [line width=1pt] (8) to (7);
		\draw [line width=1pt] (7) to (6);
		\draw [line width=1pt](40) to (41);
		\draw [line width=1pt](42) to (43);
	\end{pgfonlayer}
\end{tikzpicture}
\end{figure}
 \qed

We will now prove that there is a triangle in the eccentric graph of the Cartesian product of $k$ graphs if and only if there is a triangle in the eccentric graph of each of the individual graphs.
\bt \label{Thm:EGirth3_iff_both_EGirth3}  
Let $G_1, \ldots, G_k$ be simple connected graphs and $G
$ be their Cartesian product. Then the girth of $E(G)=3$ if and only if the girth of $E(G_i)=3$  for all $i\in [k]$.
\et
\NI \pf  
First, suppose there is a triangle in $E(G_i)$ for all $i\in [k]$.
Let $\{u_i, v_i, w_i\}$ be a triangle in $E(G_i)$ such that $e_{G_1}(u_i)\leq e_{G_1}(v_i)\leq e_{G_1}(w_i)$  for all $i\in [k]$. Therefore by \Cref{lem:eccentricity_in_cartesian_product}, $(u_1, \ldots, u_k), (v_1, \ldots, v_k), (w_1, \ldots, w_k)$ forms a triangle in $E(G)$.
Conversely, suppose $(u_1, \ldots, u_k), (v_1, \ldots, v_k), (w_1, \ldots, w_k)$ forms a triangle in $E(G)$, then again by \Cref{lem:eccentricity_in_cartesian_product} $\{u_i, v_i, w_i\}$ forms a triangle in $E(G_i)$ for all $i\in [k]$.
\qed

 \bt \label{thm:C_prod_Girth_4_if_both_girths_more_than_4}
Let $G_1, \ldots, G_k$ be simple connected graphs such that the eccentric girth of at least two of them is greater than two.
Let $G=G_1\square \cdots \square G_k$, then the girth of $E(G)$ is four except when the girth of $E(G_i)$ is exactly three for all $i\in[k]$.
\et

\NI\pf
Suppose $E(G_1)$ and $E(G_2)$ have girth greater than two and $C_1$ and $C_2$ are cycles in $E(G_1)$ and $E(G_2)$, respectively. Let
$v_1$ be a vertex of the largest eccentricity on $C_1$ and $v_2$ be a vertex of the smallest eccentricity on $C_2$. In particular, if $u_1,w_1$ are neighbours of $v_1$ in $C_1$ and $ u_2, w_2$ are neighbours of $v_2$ in $C_2$, then 
\[
e_{G_1}(v_1)\geq \max\{e_{G_1}(u_1),e_{G_1}(w_1)\} \,\text{  
 and   }\, e_{G}(v_2)\leq \min\{e_{G_2}(u_2),e_{G_2}(w_2)\}.
\]
Hence, the result follows from \Cref{Thm:EGirth3_iff_both_EGirth3} and \Cref{lem:4Cycle_in_E(G)}.
\qed

Based on the above-stated theorems, it can be concluded that the eccentric girth of the Cartesian product of graphs, in which at least two have non-zero eccentric girth, is either three or four.

\subsection{Eccentric girth of the Cartesian product of trees}
\label{sec:EccGirthOfCPofTrees}
Recall that in \cref{sec:EccGir of a tree}, we observed that the eccentric girth of a tree could either be zero, three or four. Now, we will prove that for the Cartesian product of trees, it can also be six in addition to the above values. We will now characterize completely the eccentric girth of the Cartesian product of trees and present an analogous result to \Cref{Thm: Tree_Girth}.

\bt
Let $T_1, \ldots, T_k$ be trees and $G=T_1\square \cdots \square T_k$. Then

$$\text{ the girth of }E(G)=\begin{cases}
    0 & \text{ if the girth of } E(T_i)=0  \text{ for all } i\in [k],\\
    3 & \text{ if the girth of } E(T_i)=3  \text{ for all } i\in [k],\\
    6 & \text{ if }G=T_1\square P_2\square \cdots \square  P_2 \text{ and } E(T_1) \text{ is } C_4\text{-free with girth three,}\\
    4 & \text{ otherwise.}
\end{cases}$$
\et
\NI \pf
First, assume $T_1,\dots,T_k$ are trees with eccentric girth 0. By \Cref{Thm: Tree_Girth}, there exists a unique diametrical path of odd length in $T_i$ with endpoints $u_i$ and $ v_i$ for all $i\in[k]$. Now consider the set of vertices $S=\{(x_1, \ldots, x_k): x_i\in \{
u_i, v_i\}, i\in[k]\}$ in $ V(G)$, then any vertex $u\in V(G)\setminus S$ is adjacent to exactly one vertex in the eccentric graph $E(G)$ and that neighbour lies in $S$. Also, note that any two vertices in $S$ are adjacent if and only if they differ at each component, therefore $E(G)$ is an acyclic graph with $2^{k-1}$ connected components. 


Second, only one of $T_i's$ say $T_1$ has non-zero eccentric girth. Now there are two cases, one is when at least one of $T_i$,  $i=2,\ldots, k$, is not $P_2$ and the other is $T_i=P_2$ for all $i=2,\ldots, k$. 

If suppose $T_2\neq P_2$,
and since $E(T_2)$ has girth zero, by \Cref{Thm: Tree_Girth} there exists a unique diametrical path with endpoints $u_2$ and $v_2$ 
and $u_2\sim_{E(T_2)} v_2$. 
Now, as $T_2\neq P_2$ and $E(T_2)$ is connected~\cite{EccMatrix2018}, there is a vertex $w_2$, adjacent to either $u_2$ or $v_2$, let's say $v_2\sim_{E(T_2)} w_2$. 
Clearly, $e_{T_2}(v_2)\geq \max \{e_{T_2}(u_2), e_{T_2}(w_2)\}$.
Additionally, as the girth of $E(T_1)$ is nonzero, it is possible to choose $u_1, v_1, w_1\in V(T_1)$ such that $u_1\sim_{E(T_1)} v_1$, $v_1\sim_{E(T_1)} w_1$ and $e_{T_1}(v_1)\leq \min\{e_{T_1}(u_1),e_{T_1}(w_1)\}$. Therefore by \Cref{lem:4Cycle_in_E(G)} and \Cref{Thm:EGirth3_iff_both_EGirth3}, the girth of $E(G)$ is four.

Let $T_i=P_2$ with endpoints $\{u_i, v_i\}$ for $i=2, \ldots, k$. If $E(T_1)$ contains a 4-cycle, $\{u_1,v_1,w_1,x_1\}$, then $\{(u_1,u_2,\ldots, u_k),(v_1,v_2,\ldots, v_k),(w_1,u_2,\ldots, u_k),(x_1,v_2,\ldots, v_k)\}$ forms a $4$-cycle in $E(G)$. Therefore the girth of $E(G)$ is four as $E(G)$ can not contain any odd cycle (because $T_2=P_2$).
If $E(T_1)$ doesn't contain a $4$-cycle, then by \Cref{Thm: Tree_Girth},
girth of $E(T_1)$ is $3$. Let $\{u_1, v_1, w_1\}$ be a $3$-cycle in $E(T_1)$ then $\{(u_1,u_2,\ldots, u_k),(v_1,v_2,\ldots, v_k),(w_1,u_2,\ldots, u_k),\allowbreak(u_1,v_2,\ldots, v_k),(v_1,u_2,\ldots, u_k),(w_1,v_2,\ldots, v_k)\}$ forms a $6$-cycle in $E(G)$. If $E(G)$ contains a $4$-cycle, then so is $E(T_1)$ as $T_i=P_2$ for all $i=2, \ldots, k$.

Finally, the rest of the cases follows from \Cref{Thm:EGirth3_iff_both_EGirth3,thm:C_prod_Girth_4_if_both_girths_more_than_4}.
\qed


As an illustration, We will now discuss the structure and the girth of the eccentric graph of the graphs obtained as the Cartesian product of two path graphs and two cycle graphs.
\subsection{Cartesian product of two path graphs}
An $m\times  n$ \emph{grid graph} is the Cartesian product of the path graphs $P_m$ and $P_n$, denoted as $P_m\square  P_n$. Let the vertices of $P_m\square P_n$ be $\{(i, j): 1\leq i \leq m,\, 1\leq j \leq n\}$. For the sake of simplicity in figures, we label a vertex $(i, j)$ by $(i-1)n+j$. \Cref{fig:grid}
shows the mentioned labelling for the grid graph $P_3\square P_5$.
\begin{figure}[ht]
    \centering
    \begin{tikzpicture}[x=0.65 cm, y=0.65 cm]
	\begin{pgfonlayer}{nodelayer}
		\node [style=doty, scale=0.6] (0) at (-6, 4) {};
		\node [style=doty, scale=0.6] (1) at (-4, 4) {};
		\node [style=doty, scale=0.6] (2) at (-2, 4) {};
		\node [style=doty, scale=0.6] (3) at (0, 4) {};
		\node [style=doty, scale=0.6] (4) at (2, 4) {};
		\node [style=doty, scale=0.6] (5) at (-6, 2) {};
		\node [style=doty, scale=0.6] (6) at (-6, 0) {};
		\node [style=doty, scale=0.6] (7) at (-4, 2) {};
		\node [style=doty, scale=0.6] (8) at (-4, 0) {};
		\node [style=doty, scale=0.6] (9) at (-2, 2) {};
		\node [style=doty, scale=0.6] (10) at (-2, 0) {};
		\node [style=doty, scale=0.6] (11) at (0, 2) {};
		\node [style=doty, scale=0.6] (12) at (0, 0) {};
		\node [style=doty, scale=0.6] (13) at (2, 2) {};
		\node [style=doty, scale=0.6] (14) at (2, 0) {};
		\node [style=none] (15) at (-6, -0.45) {$1$};
		\node [style=none] (16) at (-4, -0.45) {$2$};
		\node [style=none] (17) at (-2, -0.45) {$3$};
		\node [style=none] (18) at (0, -0.45) {$4$};
		\node [style=none] (19) at (2, -0.45) {$5$};
		\node [style=none] (20) at (-6.35, 2) {$6$};
		\node [style=none] (21) at (-4.5, 1.75) {$7$};
		\node [style=none] (22) at (-2.5, 1.75) {$8$};
		\node [style=none] (23) at (-0.5, 1.75) {$9$};
		\node [style=none] (24) at (2.4, 2) {$10$};
		\node [style=none] (25) at (-6, 4.45) {$11$};
		\node [style=none] (26) at (-4, 4.45) {$12$};
		\node [style=none] (27) at (-2, 4.45) {$13$};
		\node [style=none] (28) at (0, 4.45) {$14$};
		\node [style=none] (29) at (2, 4.45) {$15$};
	\end{pgfonlayer}
	\begin{pgfonlayer}{edgelayer}
		\draw[line width=1pt] (0) to (4);
		\draw[line width=1pt] (4) to (14);
		\draw[line width=1pt] (14) to (6);
		\draw[line width=1pt] (6) to (0);
		\draw[line width=1pt] (5) to (13);
		\draw[line width=1pt] (3) to (12);
		\draw[line width=1pt] (2) to (10);
		\draw[line width=1pt] (1) to (8);
	\end{pgfonlayer}
\end{tikzpicture}
    \caption{The grid graph, $P_3\square P_5$ }
    \label{fig:grid}
\end{figure}
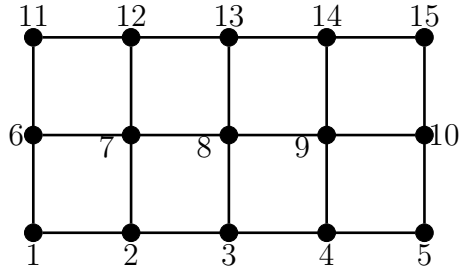

Let $G=P_m\square P_n$ be a grid. Then the eccentricity of the vertices is given by
$$e\big((i,j)\big)=\begin{cases}
d\big((i,j), (m, n)\big) & \text{~if~} 1\leq i\leq \lceil\frac{m}{2}\rceil,\;1\leq j\leq \lceil\frac{n}{2}\rceil,\\
d\big((i,j), (1,1)\big) & \text{~if~} \lfloor \frac{m}{2}\rfloor< i\leq m,\; \lfloor\frac{n}{2}\rfloor< j\leq n,\\
d\big((i,j), (m, 1)\big) & \text{~if~} 1\leq i\leq \lceil\frac{m}{2}\rceil,\; \lfloor\frac{n}{2}\rfloor< j\leq n,\\
d\big((i,j), (1, n)\big) & \text{~if~} \lfloor\frac{m}{2}\rfloor< i\leq m,\;1\leq j\leq \lceil\frac{n}{2}\rceil.
\end{cases}$$
Note that $(1, 1), (1, n), (m, 1)$ and $(m, n)$ have the maximum eccentricity, which is $m+n$. Therefore,

$$(i,j) \sim_{E(G)} \begin{cases}
(m, n) & \text{~if~}  1\leq i\leq \lceil\frac{m}{2}\rceil,\;1\leq j\leq \lceil\frac{n}{2}\rceil,\\
 (1,1) & \text{~if~} \lfloor \frac{m}{2}\rfloor< i\leq m,\; \lfloor\frac{n}{2}\rfloor< j\leq n,\\
(m, 1) & \text{~if~} 1\leq i\leq \lceil\frac{m}{2}\rceil,\; \lfloor\frac{n}{2}\rfloor< j\leq n,\\
(1, n) & \text{~if~} \lfloor\frac{m}{2}\rfloor< i\leq m,\;1\leq j\leq \lceil\frac{n}{2}\rceil.
\end{cases}$$

From the above adjacency relations, it is clear that the eccentric graph of $P_m\square P_n$ has a specific structure depending on the parity of $m$ and $n$. \emph{Further, note that the girth of the eccentric graph $E(P_m\square P_n)$ is zero if both $m$ and $n$ are even, four if exactly one of $m$ and $n$ is even, and three if both $m$ and $n$ are odd.} Illustrations for all three cases are provided in \Cref{fig:Ecc_graphs_of_Grids}.
\begin{figure}[H]
    \centering
    \begin{tikzpicture}[x=0.75 cm, y=0.75 cm]
	\begin{pgfonlayer}{nodelayer}
		\node [style=doty, scale=0.6] (0) at (2, 2.5) {};
		\node [style=doty, scale=0.6] (1) at (6, 2.5) {};
		\node [style=doty, scale=0.6] (2) at (2, 0) {};
		\node [style=doty, scale=0.6] (3) at (6, 0) {};
		\node [style=doty, scale=0.6] (4) at (4, 3.5) {};
		\node [style=doty, scale=0.6] (5) at (4, 4.5) {};
		\node [style=doty, scale=0.6] (6) at (4, 2.5) {};
		\node [style=doty, scale=0.6] (7) at (4, 0) {};
		\node [style=doty, scale=0.6] (8) at (4, -1) {};
		\node [style=doty, scale=0.6] (9) at (4, -2) {};
		\node [style=doty, scale=0.6] (10) at (2, 4.5) {};
		\node [style=doty, scale=0.6] (11) at (1.25, 4.25) {};
		\node [style=doty, scale=0.6] (12) at (0.5, 4) {};
		\node [style=doty, scale=0.6] (13) at (0.25, 3.25) {};
		\node [style=doty, scale=0.6] (14) at (0, 2.5) {};
		\node [style=doty, scale=0.6] (15) at (0, 0) {};
		\node [style=doty, scale=0.6] (16) at (0.25, -0.75) {};
		\node [style=doty, scale=0.6] (17) at (0.5, -1.5) {};
		\node [style=doty, scale=0.6] (18) at (1.25, -1.75) {};
		\node [style=doty, scale=0.6] (19) at (2, -2) {};
		\node [style=doty, scale=0.6] (20) at (6, 4.5) {};
		\node [style=doty, scale=0.6] (21) at (6.75, 4.25) {};
		\node [style=doty, scale=0.6] (22) at (7.5, 4) {};
		\node [style=doty, scale=0.6] (23) at (7.75, 3.25) {};
		\node [style=doty, scale=0.6] (24) at (8, 2.5) {};
		\node [style=doty, scale=0.6] (25) at (8, 0) {};
		\node [style=doty, scale=0.6] (26) at (7.75, -0.75) {};
		\node [style=doty, scale=0.6] (27) at (7.5, -1.5) {};
		\node [style=doty, scale=0.6] (28) at (6.75, -1.75) {};
		\node [style=doty, scale=0.6] (29) at (6, -2) {};
		\node [style=none] (30) at (1.75, 2.25) {};
		\node [style=none] (31) at (1.75, 0.25) {1};
		\node [style=none] (32) at (6.38, 0.25) {25};
		\node [style=none] (33) at (1.75, 2.25) {30};
		\node [style=none] (34) at (6.25, 2.25) {6};
		\node [style=none] (35) at (4, 5) {13};
		\node [style=none] (36) at (4, 4) {14};
		\node [style=none] (40) at (4, 3) {15};
		\node [style=none] (41) at (4, -0.5) {16};
		\node [style=none] (42) at (4, -1.5) {17};
		\node [style=none] (43) at (4, -2.5) {18};
		\node [style=none] (44) at (-0.5, 0) {22};
		\node [style=none] (45) at (-0.25, -1) {23};
		\node [style=none] (46) at (0, -1.75) {24};
		\node [style=none] (47) at (1, -2.25) {28};
		\node [style=none] (48) at (2, -2.5) {29};
		\node [style=none] (49) at (6, 5) {19};
		\node [style=none] (50) at (7, 4.75) {20};
		\node [style=none] (51) at (8, 4.25) {21};
		\node [style=none] (52) at (8.25, 3.5) {26};
		\node [style=none] (53) at (8.5, 2.5) {27};
		\node [style=none] (54) at (-0.5, 2.5) {2};
		\node [style=none] (55) at (-0.25, 3.5) {3};
		\node [style=none] (56) at (0.25, 4.25) {7};
		\node [style=none] (57) at (1, 4.75) {8};
		\node [style=none] (58) at (2, 5) {9};
		\node [style=none] (59) at (8.5, 0) {4};
		\node [style=none] (60) at (8.25, -1) {5};
		\node [style=none] (61) at (8, -1.75) {10};
		\node [style=none] (62) at (7, -2.25) {11};
		\node [style=none] (63) at (6, -2.5) {12};
		\node [style=doty, scale=0.6] (100) at (-9, 2.25) {};
		\node [style=doty, scale=0.6] (102) at (-9, 0.25) {};
		\node [style=doty, scale=0.6] (110) at (-7.5, 3.75) {};
		\node [style=doty, scale=0.6] (111) at (-8.25, 4.25) {};
		\node [style=doty, scale=0.6] (112) at (-9, 4.5) {};
		\node [style=doty, scale=0.6] (113) at (-9.75, 4.25) {};
		\node [style=doty, scale=0.6] (114) at (-10.5, 3.75) {};
		\node [style=doty, scale=0.6] (115) at (-10.5, -1.25) {};
		\node [style=doty, scale=0.6] (116) at (-9.75, -1.75) {};
		\node [style=doty, scale=0.6] (117) at (-9, -2) {};
		\node [style=doty, scale=0.6] (118) at (-8.25, -1.75) {};
		\node [style=doty, scale=0.6] (119) at (-7.5, -1.25) {};
		\node [style=none] (131) at (-9.5, 0.25) {1};
		\node [style=none] (133) at (-9.5, 2.25) {24};
		\node [style=none] (144) at (-11, -1.5) {22};
		\node [style=none] (145) at (-10, -2.25) {23};
		\node [style=none] (146) at (-9, -2.5) {18};
		\node [style=none] (147) at (-8, -2.25) {17};
		\node [style=none] (148) at (-7.1, -1.4) {16};
		\node [style=none] (154) at (-11, 4) {2};
		\node [style=none] (155) at (-10, 4.75) {3};
		\node [style=none] (156) at (-9, 5) {7};
		\node [style=none] (157) at (-8, 4.54) {8};
		\node [style=none] (158) at (-7.1, 4) {9};
		\node [style=doty, scale=0.6] (159) at (-4.25, 2.25) {};
		\node [style=doty, scale=0.6] (160) at (-4.25, 0.25) {};
		\node [style=doty, scale=0.6] (161) at (-2.75, 3.75) {};
		\node [style=doty, scale=0.6] (162) at (-3.5, 4.25) {};
		\node [style=doty, scale=0.6] (163) at (-4.25, 4.5) {};
		\node [style=doty, scale=0.6] (164) at (-5, 4.25) {};
		\node [style=doty, scale=0.6] (165) at (-5.75, 3.75) {};
		\node [style=doty, scale=0.6] (166) at (-5.75, -1.25) {};
		\node [style=doty, scale=0.6] (167) at (-5, -1.75) {};
		\node [style=doty, scale=0.6] (168) at (-4.25, -2) {};
		\node [style=doty, scale=0.6] (169) at (-3.5, -1.75) {};
		\node [style=doty, scale=0.6] (170) at (-2.75, -1.25) {};
		\node [style=none] (171) at (-4.75, 0.25) {6};
		\node [style=none] (172) at (-4.75, 2.25) {19};
		\node [style=none] (173) at (-6.15, -1.4) {13};
		\node [style=none] (174) at (-5.25, -2.25) {14};
		\node [style=none] (175) at (-4.25, -2.5) {15};
		\node [style=none] (176) at (-3.25, -2.25) {20};
		\node [style=none] (177) at (-2.5, -1.75) {21};
		\node [style=none] (178) at (-6.1, 4) {4};
		\node [style=none] (179) at (-5.25, 4.5) {5};
		\node [style=none] (180) at (-4.25, 5) {10};
		\node [style=none] (181) at (-3.25, 4.75) {11};
		\node [style=none] (182) at (-2.25, 4) {12};
		\node [style=none] (183) at (-7, -4) {(a) $E(P_4\square P_6)$};
		\node [style=none] (184) at (4, -4) {(b)  $E(P_5\square P_6)$};
	\end{pgfonlayer}
	\begin{pgfonlayer}{edgelayer}
		\draw (0) to (2);
		\draw (0) to (4);
		\draw (4) to (1);
		\draw (0) to (5);
		\draw (5) to (1);
		\draw (1) to (6);
		\draw (6) to (0);
		\draw (1) to (3);
		\draw (3) to (7);
		\draw (7) to (2);
		\draw (2) to (8);
		\draw (8) to (3);
		\draw (3) to (9);
		\draw (9) to (2);
		\draw (10) to (0);
		\draw (0) to (14);
		\draw (13) to (0);
		\draw (12) to (0);
		\draw (11) to (0);
		\draw (20) to (1);
		\draw (21) to (1);
		\draw (22) to (1);
		\draw (23) to (1);
		\draw (24) to (1);
		\draw (3) to (25);
		\draw (3) to (26);
		\draw (3) to (27);
		\draw (3) to (28);
		\draw (3) to (29);
		\draw (2) to (15);
		\draw (2) to (16);
		\draw (2) to (17);
		\draw (2) to (18);
		\draw (2) to (19);
		\draw (100) to (102);
		\draw (110) to (100);
		\draw (100) to (114);
		\draw (113) to (100);
		\draw (112) to (100);
		\draw (111) to (100);
		\draw (102) to (115);
		\draw (102) to (116);
		\draw (102) to (117);
		\draw (102) to (118);
		\draw (102) to (119);
		\draw (159) to (160);
		\draw (161) to (159);
		\draw (159) to (165);
		\draw (164) to (159);
		\draw (163) to (159);
		\draw (162) to (159);
		\draw (160) to (166);
		\draw (160) to (167);
		\draw (160) to (168);
		\draw (160) to (169);
		\draw (160) to (170);
	\end{pgfonlayer}
\end{tikzpicture}

\vspace{1.0 cm}

\begin{tikzpicture}[x=0.75 cm,y=0.75 cm]
	\begin{pgfonlayer}{nodelayer}
		\node [style=doty, scale=0.6] (35) at (4, 2) {};
		\node [style=doty, scale=0.6] (36) at (6, 2) {};
		\node [style=doty, scale=0.6] (37) at (8, 2) {};
		\node [style=doty, scale=0.6] (38) at (6, 4) {};
		\node [style=doty, scale=0.6] (39) at (6, 0) {};
		\node [style=doty, scale=0.6] (40) at (4, 4) {};
		\node [style=doty, scale=0.6] (41) at (3, 5) {};
		\node [style=doty, scale=0.6] (42) at (2, 6) {};
		\node [style=doty, scale=0.6] (43) at (8, 0) {};
		\node [style=doty, scale=0.6] (44) at (9, -1) {};
		\node [style=doty, scale=0.6] (45) at (10, -2) {};
		\node [style=doty, scale=0.6] (46) at (4, 0) {};
		\node [style=doty, scale=0.6] (47) at (3, -1) {};
		\node [style=doty, scale=0.6] (48) at (8, 4) {};
		\node [style=doty, scale=0.6] (49) at (9, 5) {};
		\node [style=doty, scale=0.6] (50) at (1.5, 3.5) {};
		\node [style=doty, scale=0.6] (51) at (1.5, 2.75) {};
		\node [style=doty, scale=0.6] (52) at (1.5, 2) {};
		\node [style=doty, scale=0.6] (53) at (1.5, 1.25) {};
		\node [style=doty, scale=0.6] (54) at (1.5, 0.5) {};
		\node [style=doty, scale=0.6] (55) at (10.5, 3.5) {};
		\node [style=doty, scale=0.6] (56) at (10.5, 2.75) {};
		\node [style=doty, scale=0.6] (57) at (10.5, 2) {};
		\node [style=doty, scale=0.6] (58) at (10.5, 1.25) {};
		\node [style=doty, scale=0.6] (59) at (10.5, 0.5) {};
		\node [style=doty, scale=0.6] (60) at (4.5, 6.5) {};
		\node [style=doty, scale=0.6] (61) at (5.25, 6.5) {};
		\node [style=doty, scale=0.6] (62) at (6, 6.5) {};
		\node [style=doty, scale=0.6] (63) at (6.75, 6.5) {};
		\node [style=doty, scale=0.6] (64) at (7.5, 6.5) {};
		\node [style=doty, scale=0.6] (65) at (4.5, -2.5) {};
		\node [style=doty, scale=0.6] (66) at (5.25, -2.5) {};
		\node [style=doty, scale=0.6] (67) at (6, -2.5) {};
		\node [style=doty, scale=0.6] (68) at (6.75, -2.5) {};
		\node [style=doty, scale=0.6] (69) at (7.5, -2.5) {};
		\node [style=none] (106) at (4.25, 1.75) {1};
		\node [style=none] (107) at (5.75, 1.75) {18};
		\node [style=none] (108) at (7.75, 1.75) {35};
		\node [style=none] (109) at (5.75, 0.25) {7};
		\node [style=none] (110) at (5.75, 3.75) {29};
		\node [style=none] (111) at (1, 3.5) {26};
		\node [style=none] (112) at (1, 2.75) {27};
		\node [style=none] (113) at (1, 2) {28};
		\node [style=none] (114) at (1, 1.25) {33};
		\node [style=none] (115) at (1, 0.5) {34};
		\node [style=none] (116) at (11, 3.5) {2};
		\node [style=none] (117) at (11, 2.75) {3};
		\node [style=none] (118) at (11, 2) {8};
		\node [style=none] (119) at (11, 1.25) {9};
		\node [style=none] (120) at (11, 0.5) {10};
		\node [style=none] (121) at (4.5, -3) {22};
		\node [style=none] (122) at (5.25, -3) {23};
		\node [style=none] (123) at (6, -3) {24};
		\node [style=none] (124) at (6.75, -3) {30};
		\node [style=none] (125) at (7.5, -3) {31};
		\node [style=none] (126) at (4.5, 7) {5};
		\node [style=none] (127) at (5.25, 7) {6};
		\node [style=none] (128) at (6, 7) {12};
		\node [style=none] (129) at (6.75, 7) {13};
		\node [style=none] (130) at (7.5, 7) {14};
		\node [style=none] (131) at (3.75, -0.25) {25};
		\node [style=none] (132) at (2.75, -1.25) {32};
		\node [style=none] (133) at (8.25, 4.25) {4};
		\node [style=none] (134) at (9.25, 5.25) {11};
		\node [style=none] (135) at (1.75, 6.25) {21};
		\node [style=none] (136) at (2.75, 5.25) {20};
		\node [style=none] (137) at (3.75, 4.25) {19};
		\node [style=none] (138) at (8.25, -0.25) {15};
		\node [style=none] (139) at (9.25, -1.25) {16};
		\node [style=none] (140) at (10.25, -2.25) {17};
        \node [style=none] (141) at (6, -4.3) {(c) $E(P_5\square P_7)$};
	\end{pgfonlayer}
	\begin{pgfonlayer}{edgelayer}
		\draw [bend left=45] (38) to (39);
		\draw [bend left, looseness=1.25] (35) to (37);
		\draw (38) to (39);
		\draw (37) to (35);
		\draw (40) to (38);
		\draw (40) to (35);
		\draw (41) to (35);
		\draw (41) to (38);
		\draw (42) to (38);
		\draw (42) to (35);
		\draw (37) to (43);
		\draw (43) to (39);
		\draw (37) to (44);
		\draw (44) to (39);
		\draw (37) to (45);
		\draw (45) to (39);
		\draw (35) to (46);
		\draw (46) to (39);
		\draw (35) to (47);
		\draw (47) to (39);
		\draw (48) to (38);
		\draw (48) to (37);
		\draw (49) to (38);
		\draw (49) to (37);
		\draw (60) to (38);
		\draw (61) to (38);
		\draw (62) to (38);
		\draw (63) to (38);
		\draw (64) to (38);
		\draw (50) to (35);
		\draw (51) to (35);
		\draw (52) to (35);
		\draw (53) to (35);
		\draw (54) to (35);
		\draw (39) to (65);
		\draw (39) to (66);
		\draw (39) to (67);
		\draw (39) to (68);
		\draw (39) to (69);
		\draw (37) to (55);
		\draw (37) to (56);
		\draw (37) to (57);
		\draw (37) to (58);
		\draw (37) to (59);
	\end{pgfonlayer}
\end{tikzpicture}
    \caption{Eccentric graphs of different grid graphs. }
    \label{fig:Ecc_graphs_of_Grids}
\end{figure}

\subsection{Cartesian product of two cycle graphs}

As discussed in \Cref{sec: basicEG},
$E(C_n)$ is isomorphic to the $\frac{n}{2}$ copies of $K_2$ for an even $n$.
Thus when $n$ and $m$ both are even, each vertex in $E(C_n\square C_m)$ has degree 1. In other words, $E(C_n\square C_m)$ is isomorphic to a graph containing $\frac{nm}{2}$ copies of $K_2$. 

For an even $n$ and an odd $m$, each vertex in $E(C_n)$ and $E(C_m)$ has degree $1$ and $2$ respectively. Therefore, $E(C_n\square C_m)$ is a $2$-regular graph. Consequently, $E(C_n\square C_m)$ is either a cycle or a union of cycles. Moreover, $E(C_n\square C_m)$ consists $\frac{n}{2}$ cycles of length $2m$ namely,
\[
\bigg(
(i,1),\Big(\frac{n}{2}+i,2\Big),
\dots,(i,m),\Big(\frac{n}{2}+i,1\Big),(i,2),\dots, \Big(\frac{n}{2}+i,m\Big)
\bigg)
\]
for $i\in[\frac{n}{2}]$. \Cref{fig:Eccofoldevencycle} shows the eccentric graph of the Cartesian product of $C_4$ and $C_3$.
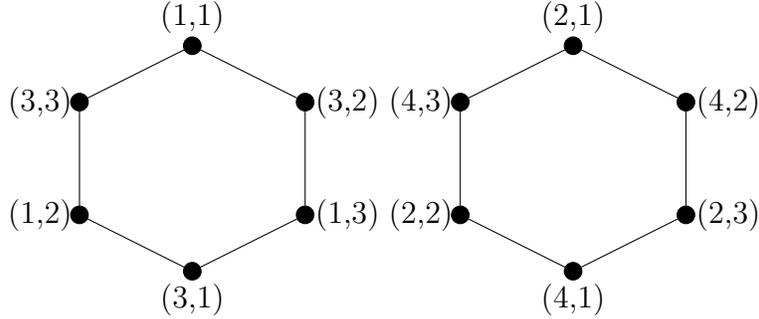
\begin{figure}[H]
    \centering
    \begin{tikzpicture}[x=0.75 cm, y=0.75 cm]
	\begin{pgfonlayer}{nodelayer}
		\node [style=doty, scale=0.6] (0) at (0, 3) {};
		\node [style=doty, scale=0.6] (1) at (2, 2) {};
		\node [style=doty, scale=0.6] (2) at (-2, 2) {};
		\node [style=doty, scale=0.6] (3) at (2, 0) {};
		\node [style=doty, scale=0.6] (4) at (-2, 0) {};
		\node [style=doty, scale=0.6] (5) at (0, -1) {};
		\node [style=none] (6) at (0, 3.5) {(1,1)};
		\node [style=none] (7) at (2.75, 2) {(3,2)};
		\node [style=none] (8) at (2.75, 0) {(1,3)};
		\node [style=none] (9) at (0, -1.5) {(3,1)};
		\node [style=none] (10) at (-2.7, 0) {(1,2)};
		\node [style=none] (11) at (-2.7, 2) {(3,3)};
		\node [style=doty, scale=0.6] (12) at (6.75, 3) {};
		\node [style=doty, scale=0.6] (13) at (8.75, 2) {};
		\node [style=doty, scale=0.6] (14) at (4.75, 2) {};
		\node [style=doty, scale=0.6] (15) at (8.75, 0) {};
		\node [style=doty, scale=0.6] (16) at (4.75, 0) {};
		\node [style=doty, scale=0.6] (17) at (6.75, -1) {};
		\node [style=none] (18) at (6.75, 3.5) {(2,1)};
		\node [style=none] (19) at (9.5, 2) {(4,2)};
		\node [style=none] (20) at (9.5, 0) {(2,3)};
		\node [style=none] (21) at (6.75, -1.5) {(4,1)};
		\node [style=none] (22) at (4.05, 0) {(2,2)};
		\node [style=none] (23) at (4.05, 2) {(4,3)};
	\end{pgfonlayer}
	\begin{pgfonlayer}{edgelayer}
		\draw (0) to (1);
		\draw (1) to (3);
		\draw (3) to (5);
		\draw (5) to (4);
		\draw (4) to (2);
		\draw (2) to (0);
		\draw (12) to (13);
		\draw (13) to (15);
		\draw (15) to (17);
		\draw (17) to (16);
		\draw (16) to (14);
		\draw (14) to (12);
	\end{pgfonlayer}
\end{tikzpicture}
    \caption{{Eccentric graph of the Cartesian product of $C_4$ and $C_3$.}}
    \label{fig:Eccofoldevencycle}
\end{figure}

When $m$ and $n$ both are $3$ the eccentric graph of $C_n\square C_m$ is shown in \Cref{fig:EccC3boxC3} and its girth is $3$ by \Cref{Thm:EGirth3_iff_both_EGirth3}, which can be seen in the figure as well.
\begin{figure}[H]
    \centering
   \begin{tikzpicture}[x=0.7 cm,y=0.7 cm]
	\begin{pgfonlayer}{nodelayer}
		\node [style=doty,  scale=0.6] (18) at (-2, 4) {};
		\node [style=doty,  scale=0.6] (19) at (-2, -2) {};
		\node [style=doty,  scale=0.6] (20) at (-3, 1) {};
		\node [style=doty,  scale=0.6] (21) at (4, 4) {};
		\node [style=doty,  scale=0.6] (22) at (4, -2) {};
		\node [style=doty,  scale=0.6] (23) at (3, 1) {};
		\node [style=doty,  scale=0.6] (24) at (0, 0) {};
		\node [style=doty,  scale=0.6] (25) at (1, -3) {};
		\node [style=doty,  scale=0.6] (26) at (1, 3) {};
		\node [style=none] (27) at (-2.75, 4) {(3,1)};
		\node [style=none] (28) at (-2.75, -1.75) {(2,3)};
		\node [style=none] (29) at (-3.75, 1) {(1,2)};
		\node [style=none] (30) at (2.45, 1.25) {(2,1)};
		\node [style=none] (31) at (-0.75, -0.25) {(3,3)};
		\node [style=none] (32) at (1, -3.5) {(1,1)};
		\node [style=none] (33) at (0.25, 2.75) {(2,2)};
		\node [style=none] (34) at (4.75, 4) {(1,3)};
		\node [style=none] (35) at (4.75, -2) {(3,2)};
	\end{pgfonlayer}
	\begin{pgfonlayer}{edgelayer}
		\draw (18) to (20);
		\draw (20) to (19);
		\draw (19) to (18);
		\draw (21) to (23);
		\draw (23) to (22);
		\draw (21) to (22);
		\draw (18) to (21);
		\draw (21) to (26);
		\draw (26) to (18);
		\draw (19) to (22);
		\draw (22) to (25);
		\draw (25) to (19);
		\draw (20) to (23);
		\draw (23) to (24);
		\draw (24) to (20);
		\draw (26) to (24);
		\draw (24) to (25);
		\draw (25) to (26);
	\end{pgfonlayer}
\end{tikzpicture}

    \caption{Eccentric graph of the Cartesian product of a 3-cycle with itself.}
    \label{fig:EccC3boxC3}
\end{figure}
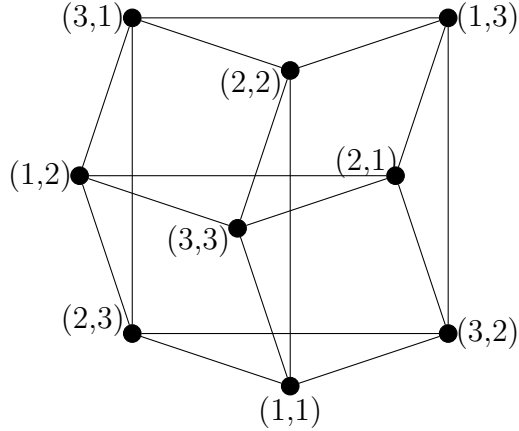

Finally, for the remaining case,
it follows from \Cref{thm:C_prod_Girth_4_if_both_girths_more_than_4} that the eccentric girth of $E(C_n\square C_m)$ is four.

The following statement summarizes the above discussion:
The eccentric girth of the Cartesian product of two cycle graphs is even except when both cycles are triangles. Moreover, 
\begin{equation*}
  \text{Eccentric girth of }E(C_n\square C_m)=\begin{cases}
    0 &\text{if both }n \text{ and } m \text{ are even,}\\
    3 &\text{if }n,m=3,\\
    2m&\text{if }n \text{ is even and } m \text{ is odd,}\\
    4 &\text{otherwise.}\\
    \end{cases}
\end{equation*}

{We will end this section with the following observation.}

\bp
For an odd value of $n$, $E(C_n\square C_n)$ is isomorphic to $C_n\square C_n$.
\ep
\NI\pf  By \Cref{cor: sameEccentricityProduct}, it is enough to show that $C_n\square\; C_n $ is isomorphic to $C_n \times C_n$ for an odd $n$. We assume the natural labelling on the vertices of $C_n$. Now, we define an isomorphism $f$ from $C_n\square\; C_n $  to $C_n \times C_n$ as follows
\begin{align*}
    f\big((1,1)\big)&= (1,1),\\
    f\big((i,1)\big)&= (n+2-i,n+2-i)\text{ for } i=2,\ldots,n,\\
    f\big((i,j)\big)&= \big[f\big((i,1)\big)+(j-1,1-j)\big]\,(\text{mod}\; n).
\end{align*}
We will write $0$ as $n$ in the computation of $f$.
To see $f$ is a bijection, first note that $f\big((i,1)\big)\neq f\big((j,1)\big)$ for $i\neq j$. Now assume $(i,j)\neq (k,l)$, this happens in either of three cases,
$(a)$ $i\neq k$ and $j=l$, $(b)$ $i=k$ and $j\neq l$, or $(c)$ $i\neq k$ and $j\neq l$. 

Consider the first case  $i\neq k$ and $j=l$ and let $f\big((i,1)\big)=(s,s)$ and $f\big((k,1)\big)=(t,t)$, clearly $s\neq t$. Now, if $f\big((i,j)\big)= f\big((k,l)\big)$ this implies $s+j-1 \equiv t+j-1 (\text{mod}\;{n})$ which leads to $s=t$, a contradiction. Therefore $f\big((i,j)\big)\neq  f\big((k,l)\big)$. Similarly, we can show for the second case. Now consider the third case $i\neq k$ and $j\neq l$, and again let $f\big((i,1)\big)=(s,s)$ and $f\big((k,1)\big)=(t,t)$, clearly $s\neq t$. Now, if $f\big((i,j)\big)= f\big((k,l)\big)$ this implies $s+j-1 \equiv t+l-1 (\text{mod}\;{n})$ and  $s+1-j \equiv t+1-l (\text{mod}\;{n})$, compatibility with addition of congruence leads to again $s=t$ (because $n$ is odd), a contradition. Therefore, $f$ is a bijection.

Now, let $(i,j)\in V(C_n \square C_n)$ and $f\big((i,j)\big)= (s,t)$. Then $f\big((i\pm 1,j)\big)=(s\pm 1, t\pm 1)(\text{mod}\;{n})$ and $f\big((i,j\pm 1)\big)=(s\pm 1, t\mp 1)(\text{mod}\;{n})$. This proves that $f$ preserves the adjacency. 
\qed

\section{Invertibilty of eccentricity matrix of the Cartesian product of trees}
\label{sec:Invertibilty of EG}

In this section, we will focus on the invertibility of the eccentric matrix for the Cartesian product of trees. First, recall the definition of the \emph{Kronecker product} of two matrices.
\bd
\label{def:KronProdMat}
Let $A=(a_{i,j})$ be an $m\times n$ matrix and $B=(b_{i,j})$ be a $p\times q$ matrix, then the \emph{Kronecker product}, $A \otimes B$, is an $mp\times nq$ block matrix defined as
\[
\begin{pmatrix}
    b_{11}A &\cdots&b_{1n}A \\ 
    \vdots &\ddots&\vdots \\ 
    b_{m1}A &\cdots&b_{mn}A \\ 
\end{pmatrix}.
\]
\ed
Kronecker product of two matrices is non-commutative in general. If $A$ and $B$ are square matrices of order $n$ and $p$, respectively, then $$\det A\otimes B=(\det A )^p(\det B)^n.$$

\bl
    \label{lem:P2sBoxStar}
    Let $T$ be a tree that is not a star or $P_4$, then the eccentricity matrix of $T\square \underbrace{ P_2\square \cdots \square P_2}_{k-1}$ is not invertible.
\el
\NI\text{\bf Proof:} Let $G=T\square P_2\square\cdots \square P_2$ and the $i^{th}$ graph in this product is the path $P_2$ with endpoints $\{u_i, v_i\}$ for $i=2,\ldots,k$. 
Note that a vertex $(x_1,x_2\ldots,x_k)$ is adjacent to $(u_1,u_2,\ldots,u_k)$ in $E(G)$ if and only if $x_i=v_i$ for $i=2,\ldots,k$ and either $x_1$ is eccentric to $u_1$ in $T_1$ or $u_1$ is eccentric to $x_1$ in $T_1$. In other words, adjacency with $(u_1,u_2,\ldots,u_k)$ in $E(G)$
solely depends on the adjacency of $u_1$ in $E(T_1)$. Now we consider three cases.

\NI \emph{Case 1}: Let the diameter of $T_1$ be $3$ and $P= a_1\,b_1\,c_1\,d_1$ be a diametrical path in $T_1$. As $T_1\neq P_4$, there must be a leaf vertex, say $e_1$, adjacent to either $b_1$ or $c_1$. Let's assume $e_1$ is adjacent to $b_1$. Now we claim that $N_{E(G)}\big((a_1,u_2,\ldots,u_k)\big)=N_{E(G)}\big((e_1,u_2,\ldots,u_k)\big)$. If a vertex $f_1$, is eccentric to $a_1$ then $f_1$ is also eccentric to $e_1$ because $d_{T_1}(a_1,f_1)=d_{T_1}(e_1,f_1)$, and if $a_1$ is eccentric to some vertex $f_1$ then so is $e_1$ because $d_{T_1}(a_1,f_1)=d_{T_1}(e_1,f_1)$. This proves our claim and hence the rows corresponding to these two vertices in $\mathcal{E}_G$ are exactly the same and therefore $\det(\mathcal{E}_G)=0$. 

\NI \emph{Case 2}: Let the diameter of $T_1$ be $4$ and $P= a_1\,b_1\,c_1\,d_1\, e_1$ be a diametrical path in $T_1$. Let $\{b_1,d_1,p_1,\ldots, p_{\ell}\}$ be the set of neighbours of $c_1$.
Note that if a vertex $x$ is eccentric to a neighbour of $c_1$ then it is also eccentric to $c_1$. 
Further, note that none of $c_1$ or its neghbours can be eccentric to any vertex in $T_1$. Therefore, row corresponding to $(c_1,u_2,\ldots,u_k)$ in the matrix $\mathcal{E}_G$ is a constant multiple of the sum of the rows corresponding to $(b_1,u_2,\ldots,u_k),(d_1,u_2,\ldots,u_k), (p_1,u_2,\ldots,u_k),\dots (p_{\ell},u_2,\ldots,u_k)$.

\NI \emph{Case 3}: Let the diameter of $T_1$ be greater than $4$ and $P= a_1\,b_1\,c_1\,d_1\ldots\, z_1$ be a diametrical path in $T_1$. By using similar arguments as in case 1 and case 2, we get the rows corresponding to $(b_1,u_2,\ldots,u_k)$ and $(c_1,u_2,\ldots,u_k)$ in $\mathcal{E}_G$ are constant multiple of each other and hence $\det(\mathcal{E}_G)=0$.
\qed

Now, we will present the main result of this section.
\bt
\label{thm:Inv of prod of Trees}
    Let $T_1, \dots, T_k$ be trees and  $G \,(=T_1\square \cdots \square T_k)$ be their Cartesian product. Then the eccentricity matrix of G, $\mathcal{E}_G$, is invertible if and only if one of them is a star or $P_4$ and the rest are $P_2$.
\et
\NI\text{\bf Proof:}
Let $T_1, \dots, T_k$ be trees with at least two vertices and  $G=T_1\square \cdots \square T_k$.
    Assume that $T_1$ is a star on $n+1$ vertices and $T_2,\dots ,T_k=P_2$. Then the eccentricity matrix of $G
    $ is 
    \[
\mathcal{E}_G=\begin{pmatrix}
    0&k+1 & k+1&\cdots &k+1\\
    k+1& 0&k+2 &\cdots &k+2\\
    \vdots&\vdots & \ddots&\\
    k+1&k+2 &\cdots & 0&k+2\\
    k+1&k+2 &\cdots &k+2 &0
\end{pmatrix} \otimes J_{2^k},
    \]
where, $J_{2^k}$ is a $2^k\times 2^k$ antidiagonal matrix with all antidiagonal entries as 1.\\
\vspace{0.2cm}\\
Note that $\det \begin{pmatrix}
    0&k+1 & k+1&\cdots &k+1\\
    k+1& 0&k+2 &\cdots &k+2\\
    \vdots&\vdots & \ddots&\\
    k+1&k+2 &\cdots & 0&k+2\\
    k+1&k+2 &\cdots &k+2 &0
\end{pmatrix}$ is $(-1)^{n+1}n(k+1)^2(k+2)^{n-1}$, also $\det J_{2^k} \neq 0$. Therefore $\det \mathcal{E}_G\neq 0$.

Now if $T_1=P_4$, then the eccentricity matrix of $G
$ is 
    \[
\mathcal{E}_G=\begin{pmatrix}
    0&0& 2+k &3+k\\
    0& 0&0 &2+k\\
    2+k&0 & 0&0\\
    3+k&2+k&0 &0
\end{pmatrix} \otimes J_{2^k},
    \]
Again, $\det \mathcal{E}_G\neq 0$, as $\det \begin{pmatrix}
    0&0& 2+k &3+k\\
    0& 0&0 &2+k\\
    2+k&0 & 0&0\\
    3+k&2+k&0 &0
\end{pmatrix}=(k+2)^4$.



For the converse part, let $T_1$ be neither a star nor $P_4$. Thus the diameter of $T_1>2$ and let $P=u_1u_2\dots u_s$ be a diametrical path in $T_1$. If each of $T_2,\dots ,T_k$ contains only pendant vertices, then the conclusion follows from \Cref{lem:P2sBoxStar}.
Therefore, we can assume without loss of generality, $T_2$ has a non-pendant vertex $v$.
Now we want to show that $\det \mathcal{E}_{G}$ is zero. This assertion holds if we can show in general $\det \mathcal{E}_{K}$ is zero, where $K$ is the Cartesian product of $T_1$, $T_2$ and a simple connected graph $H$.  
Let $(u_i, v, x)\in V(K)$. Note that $(u_i, v, x)$ cannot be farthest from (and hence, eccentric to) any vertex in $K$ {because $v$ is a non-pendant}. Consequently, only those vertices are adjacent to $(u_i,v,x)$ (in $E(K)$) which are 
eccentric to $(u_i,v,x)$.
Thus,
\begin{equation}
\label{eqn:EccinNonpendantCase}
N_{E(K)}(u_i,v, x)=\{(w_i,w, y): w_i,w, y \text{  are eccentric to } u_i,v,x \text{ respectively} \}.
\end{equation}

Now if any vertex is eccentric to $u_1$ in $T_1$ then the same vertex is eccentric to $u_2$ as well in $T_1$ leading to
\[
N_{E(K)}(u_1,v, x)=
N_{E(K)}(u_2,v,x).
\]
Thus, rows corresponding to $(u_1,v, x)$ in $\mathcal{E}_K$ is a constant multiple of that of $(u_2,v,x)$, proving the non-invertibility of $\mathcal{E}_K$.

\qed

\section*{Acknowledgements}Authors thank Professor Arvind Ayyer for his valuable comments.
The first author thanks the Prime Minister Research Fellowship, India, for the funding. The second author acknowledges the support of the Council of Scientific $\And$ Industrial Research, India (File number: 09/921(0347)/2021-EMR-I). 
\bibliographystyle{siam}
	\bibliography{ecc}

\end{document}